\documentclass[11pt,a4paper,english]{amsart}

\usepackage{enumitem}

\usepackage{amssymb}
\usepackage[centertags]{amsmath}
\usepackage{amsfonts}
\usepackage{amsthm}
\usepackage{color}

\usepackage{fancyhdr}
\usepackage{amssymb}

\usepackage{amsmath}  

\usepackage{amsfonts}
\usepackage{amsthm}
\usepackage{color}
\usepackage{fouridx}

\usepackage{amssymb}

\usepackage{babel}

\usepackage{mathtools}

\mathtoolsset{showonlyrefs}

\usepackage{graphicx}
\usepackage{epstopdf}

{\theoremstyle{plain}

\newtheorem{lemma}{Lemma}[section]
\newtheorem{theorem}[lemma]{Theorem}
\newtheorem{corollary}[lemma]{Corollary}
\newtheorem{proposition}[lemma]{Proposition}

} {\theoremstyle{definition}

\newtheorem{example}[lemma]{Example}
\newtheorem{remark}[lemma]{Remark}

\newtheorem{ass}[lemma]{Assumption}

}

\newcommand{\Err}{{e}}

\renewcommand{\epsilon}{\varepsilon}                 
\renewcommand{\phi}{\varphi}
\renewcommand{\theta}{\vartheta}
\renewcommand{\le}{\leqslant}
\renewcommand{\ge}{\geqslant}
\newcommand{\origsetminus}{} \let\origsetminus=\setminus     
\renewcommand{\setminus}{\!\origsetminus\!}
\newcommand{\origfoo}{} \let\origfoo=\sqrt           
\renewcommand{\sqrt}[1]{\origfoo{#1}\;}

\newcommand{\inner}[3][]{(#2,#3)_{#1}}

\newcommand{\RR}{{\mathbb{R}}}

\DeclareMathOperator*{\esssup}{ess\,sup}

\renewcommand{\Re}{\text{\rm Re}\,}                  
\newcommand{\R}{{\mathbb R}}                
\newcommand{\EE}{{\mathbb E}}

\newcommand{\calL}{\mathcal L}

\newcommand{\N}{{\mathbb{N}}}                

\renewcommand{\P}{P}

\newcommand{\D}{{\mathcal D}}

\newcounter{zahl}

\newcommand{\h}{n}                    

\newcommand{\n}{\Vert}

\newcommand{\bcase}{\begin{cases}}
\newcommand{\ecase}{\end{cases}}
\newcommand{\pmat}{\begin{pmatrix}}
\newcommand{\epmat}{\end{pmatrix}}

\newcommand{\barray}{\begin{array}{rcl}}
\newcommand{\earray}{\end{array}}

\newcommand{\del}[1]{}

\newcommand{\lqq}{\lefteqn}

\newcommand{\lk}{\left}
\newcommand{\rk}{\right}

\newcommand{\be} {\begin{enumerate} }
\newcommand{\ee} {\end{enumerate} }

\newcommand{\CF}{{ \mathcal{ F } }}

\newcommand{\CC}{{\mathbb{C}}}

\newcommand{\NN}{\mathbb N}

\newcommand{\PP}{{\mathbb{P}}}
\newcommand{\TT}{{\rm I \kern -0.2em T}}

\newcommand{\DEQS}{\begin{eqnarray*}}
\newcommand{\EEQS}{\end{eqnarray*}}
\newcommand{\DEQSZ}{\begin{eqnarray}}
\newcommand{\EEQSZ}{\end{eqnarray}}

\newcommand{\fahim}[1]{{\color{black}#1}}
\newcommand{\fahimm}[1]{{\color{black}#1}}

\begin{document}

\title[Approximation of stochastic fractional order evolution equations]{Some approximation results for mild solutions of stochastic fractional order evolution equations driven by Gaussian noise}

\author{K.~Fahim}
\address{Department of Mathematics\\
Institut Teknologi Sepuluh Nopember\\
Kampus ITS Sukolilo-Surabaya 60111, Indonesia}
\email{kfahim@matematika.its.ac.id}

\author{E.~Hausenblas}
\address{Department of Mathematics\\
Montanuniversity Leoben\\
8700 Leoben, Austria}
\email{erika.hausenblas@unileoben.ac.at}
\thanks{E. Hausenblas was partially supported by the Austrian Science Foundation (FWF) Project P 28010.}

\author{M.~Kov\'acs}
\address{Faculty of Information Technology and Bionics\\P\'azm\'any P\'eter Catholic University\\
Budapest, Hungary}
\email{kovacs.mihaly@itk.ppke.hu}
\thanks{M. Kovács acknowledges the support of the Marsden Fund of the Royal Society of New Zealand through grant. no. 18-UOO-143, the Swedish Research Council (VR) through grant no. 2017-04274 and the NKFIH through grant no. 131545. }

\date{}


\begin{abstract}
We investigate the quality of space approximation of a class of stochastic integral equations of convolution type with Gaussian noise. Such equations arise, for example, when considering
mild solutions of stochastic fractional order partial differential equations but also when considering mild solutions of classical stochastic partial differential equations.
The key requirement for the equations is a smoothing property of the deterministic evolution operator which is typical in parabolic type problems.
We show that if one has access to nonsmooth data estimates for the deterministic error operator  together with its derivative of a space discretization procedure, then one obtains error estimates in pathwise H\"older norms with rates that can be read off the deterministic error rates. We illustrate the main result by considering a class of stochastic fractional order partial differential equations and space approximations performed by spectral Galerkin methods and finite elements. We also improve an existing result on the stochastic heat equation.

%
\end{abstract}

\maketitle

\section{Introduction}

Let $H$ be a real separable Hilbert space and let  $W_H$ be a
$H$-cylindrical Wiener process on a complete, filtered probability space $(\Omega,\CF,(\CF_t)_{t\ge0},\PP)$ with respect to the filtration $(\CF_t)_{t\ge0}$. To be more precise, we assume that $(\CF_t)_{t\ge0}$ satisfies  the usual conditions, which are, $(\CF_t)_{t\ge0}$ is right-continuous and $\CF_0$ contains all $\PP$-nullsets of $\CF$.
Let $H_i$, $i=0,1,2$ be real separable Hilbert spaces to be  specified later on but they are typically associated with fractional powers of a linear operator.
We consider integral equations of the form
\begin{equation}\label{varcons0}
U(t)=X^0(t)+\int_{0}^{t}S^2(t-s)F(s,U(s))\,ds+\int_{0}^{t}S^1(t-s)G(s,U(s))\,dW_H(s). 
\end{equation} 
Here, the non-linear functions $G:[0,T]\times H_0\to L_{HS}(H,H_1)$, where $L_{HS}(H,H_1)$ denotes  the space of Hilbert-Schmidt operators form $H$ to $H_1$, and $F:[0,T]\times H_0\to H_2$
 are assumed to satisfy global Lipschitz and linear growth conditions. The operator families $S^i(t):H_i\to H_0$, $i=1,2,$ are assumed to admit certain smoothing estimates for $t>0$.

 \medskip

A typical example where the integral equation \eqref{varcons0} arises is when defining mild solutions of fractional order equations of the form \cite{DL},
\begin{equation}\label{SDE-intro}
\begin{aligned} U(t) & =  U_0+tU_1-A\int_0 ^t \frac{(t-s)^{\alpha-1}}{\Gamma(\alpha)} U(s)\,ds
\\ &{}\quad + \int_0 ^t \frac{(t-s)^{\kappa-1}}{\Gamma(\kappa)} F(s,U(s))\,ds
\\ &{}\quad + \int_0^ t \frac{(t-s)^{\beta-1}}{\Gamma(\beta)} \, G(s,U(s))\,dW_H(s);\quad t\in [0,T],
\end{aligned}
\end{equation}
see Section \ref{sg_one} for more details.
Here, $A$ is a densely defined, possibly unbounded, non-negative operator on the Hilbert space $H_0$,  $\alpha\in (0,2)$, $\beta>\frac{1}{2}$ and $\kappa>0$. The restriction $\beta>\frac12$ is needed otherwise the stochastic integral does not make sense even for constant $G$ as for $\beta\leq \frac12$ the function $t\to t^{\beta-1}$ is not square integrable on $[0,T]$.
For  $\alpha \in (0,1)$ (and $U_1=0$ in this case), equation \eqref{SDE-intro} becomes a fractional stochastic heat equation,
for $\alpha\in (1,2)$ equation \eqref{SDE-intro} becomes a fractional stochastic wave equation.

\medskip

Time fractional stochastic heat type equations might be used to model \fahim{phenomena} with random effects with thermal memory \cite{MN15}.
 In its simplest form, \fahim{the fractional stochastic} heat equation has the form
\begin{equation}\label{eq:shebase}
dU+AD_t^{1-\alpha}(U)dt=F(U)dt+G(U)dW_H(t); \,U(0)=U_0,\,\alpha\in (0,1),
\end{equation}
where
$$
D_t^{1-\alpha}(U)(t)=\frac{1}{\Gamma(\alpha)}\frac{d}{dt}\int_0^t(t-s)^{\alpha-1}U(s)\,d s,\,\alpha\in (0,1).
$$
Equation \eqref{eq:shebase}  corresponds to \eqref{SDE-intro} with $\beta=\kappa=1$, $\alpha\in (0,1)$ and $U_1=0$.

Time fractional stochastic wave type equations may be used to model random forcing effects in viscoelastic materials which exhibit a simple power-law creep behaviour \cite{CDP,KLS20}.
 The simplest form of the stochastic wave equation takes the form
\begin{equation}\label{swebase}
dU+AI^{\alpha-1}(U)dt=U_1dt+F(U)dt+G(U)dW_H(t); \,U(0)=U_0,\,\alpha\in (1,2),
\end{equation}
where
$$
I^{\alpha-1}(U)(t)=\frac{1}{\Gamma(\alpha-1)}\int_0^t(t-s)^{\alpha-2}U(s)\,ds,\, \alpha>1,
$$
and $U_1$ is the initial data for $\dot{U}$. Equation \eqref{swebase}   corresponds to  \eqref{SDE-intro} with $\beta=\kappa=1$ and $\alpha\in (1,2)$.

In both cases the parameters $\beta$ and $\kappa$ can be used to model the time-regularity of the stochastic, respectively, the deterministic feedback. For example, when $\beta<1$, then the driving process is rougher than the Wiener process, while if $\beta>1$, it is smoother. It is important to note that while the parameter choice in \eqref{SDE-intro} corresponding to $\alpha=\beta=\kappa=1$ is the standard stochastic heat equation, the parameter choice $\beta=\kappa=1$ and $\alpha= 2$ does not result in the standard stochastic wave equation but in something much more irregular as in this case the noise will drive $\dot{U}$ and not $\ddot{U}$.  The standard stochastic wave equation would correspond to the choice  $\beta=\kappa=2$ and $\alpha= 2$. This case is not covered by our paper, since in our setting  crucial estimates  has constants blowing up as $\alpha\to 2$. In particular, the constants in the fundamental regularity estimates \eqref{eq:sabe} and \eqref{eq:sabe1} will blow up, and therefore we cannot say anything for the limiting case $\alpha=2$ by taking $\alpha\to 2$.

Our aim is to approximate stochastic integro-differential equations of the type \eqref{varcons0} and derive error estimates in pathwise H\"older norms in time. To this end we consider
approximations of \eqref{varcons0} given by the following integral equation
\begin{align}\label{varconsni}
U_n(t)&=X_n^0(t)+\int_{0}^{t}S_n^2(t-s)F(s,U_n(s))\,ds+\int_{0}^{t}S_n^1(t-s)G(s,U_n(s))\,dW_H(s). 
\end{align}
Here, the approximation $U_n$ can typically be a  spatial approximation derived via a spectral Galerkin or a continuous finite element method. The main \fahim{purpose} of our work is to derive rates of convergence of the strong error over $C^\gamma([0,T];H_0)$; that is, we derive error estimates for $U-U_n$ in $L^p(\Omega;C^\gamma([0,T];H_0))$. Here,  for a function $f:[0,T]\to E$, where $E$ is a Banach space, the H\"older seminorm is defined by
$$
\|f\|_{C^\gamma([0,T];E)}=\sup_{\stackrel{t,s \in [0,T]}{t\neq s}}\frac{\|f(t)-f(s)\|_E}{|t-s|^\gamma},\,\gamma \in (0,1).
$$
In particular, we derive a rate of convergence of the strong error over $C([0,T];H_0)$ (the space of $H_0$-valued continuous functions on $[0,T]$ equipped with the supremum norm); that is, an error estimate for $U-U_n$ in $L^p(\Omega;C([0,T];H_0))$. These error estimates are derived given that one has access to deterministic nonsmooth error estimates for $S^i-S^i_n$ and $\frac{d}{dt}(S^i-S^i_n)$. The main point is that the rate of convergence in these norms can be directly read off  the deterministic error rates. While traditionally error estimates for $\frac{d}{dt}(S^i-S^i_n)$ are seldom considered they are not out of reach in many cases (see, for example, \cite[Theorem 3.4]{Thomeebook}, for finite elements for parabolic problems). We demonstrate by two examples how to obtain such estimates for fractional order equations both for spectral Galerkin and for a standard continuous finite element method. In general, when $S^i$ are resolvent families for certain parabolic integro-differential problems arising, for example, in viscoelasticity, \cite{baeumeretall,CDP,KLS20,pruss}, these nonsmooth data estimates are direct consequences of the smoothing property of the resolvent family of the linear deterministic problem, at least for the spectral Galerkin method.

Our motivation for considering estimates in such norms is twofold. Firstly, estimates with respect to the $L^p(\Omega;C([0,T];H_0))$-norm are useful for
using standard localization arguments \cite{Gyo,Pri} in order to extend approximation results for equations with globally Lipschitz
continuous nonlinearities to results for equations with nonlinearities that are only Lipschitz continuous
on bounded sets. We refer to   \cite[Section 4]{sonja2} for further details in the semigroup case. Secondly, as Remark \ref{rem:reg} shows, the processes $U$ and $U_n$ can be viewed as random variables in $L^p(\Omega;C^\gamma([0,T];H_0))$ and therefore it natural to consider the approximation error in the corresponding norm. For further applications of approximations in H\"older norms  we refer to  \cite{sonja2} and \cite{bally}.

Finally, we would like to emphasize that the derivation of error estimates in the $L^p(\Omega;C([0,T];H_0))$-norm is usually a nontrivial task even when the operator \fahimm{family} $S^1$ is a semigroup. This is because, in general, the stochastic convolution appearing  in \eqref{varcons0} fails to be a semimartingale and hence Doob’s
maximal inequality cannot be applied to obtain estimates with respect to the $L^p(\Omega;C([0,T];H_0))$-norm. In case the operator family $S^1$ is a semigroup one may employ the factorization method of Da Prato, Kwapien and Zabczyk \cite{DKZ} directly to obtain such estimates, for an instance, see, for example, \cite{KLME}. However, when the semigroup property does not hold; that is, there is a nontrivial memory effect in the equation, then even this approach fails.


\subsection*{The state of the art}
For analytical results, such as existence, uniqueness and regularity of various stochastic Volterra-type integro-differential equations driven by Wiener noise we refer to \cite{baeumeretall,Barbu14,Bona13,Bona06,Bona09a,CDP,CDH,KarLiz07,KarLiz09,KarZab00} and for results concerning asymptotic behaviour of solutions to \cite{Bona12,FN17}. The particular case of fractional order equations driven by Wiener noise are considered in \cite{Bona03,DL08,DL11a,DL,KimKim19,LotRoz20,MN15}. Various integro-differential equations driven by L\'evy noise are analysed in \cite{DGy14,HaKo18} with the particular case of  fractional order equations in \cite{Bona09}. Finally a class of linear Volterra integro-differential equations driven by fractional Brownian motion are investigated in
\cite{Sp09,SpW10}.

The main \fahim{purpose} of our work is to derive rates of convergence for space approximations. Here, we consider the strong error over $C^\gamma([0,T];H_0)$ and $C([0,T];H_0)$, that is, we derive error estimates in $L^p(\Omega;C^\gamma([0,T];H_0))$ and $L^p(\Omega;C([0,T];H_0))$. To our knowledge all existing work on the numerical analysis of stochastic fractional order differential equations are considering a much weaker error measure; that is, the error measure $\sup_{t\in [0,T]}\EE\|U(t)-U_n(t)\|_{H_0}^p$ (mainly for $p=2$), see, for example, \cite{Gunz19b,Gunz19a,JinYNZh19,mishi1,KLS20,WuYanYan20}, or the weak error \cite{AKL16,JinYNZh19,KLS15,KP14}. For similar works in the setting of abstract evolution equations without memory kernel we refer to  \cite{sonja1,sonja3}, 
 and \cite{sonja2}.

\subsection*{The structure of the paper}

The paper is organized as follows. In Section \ref{AAR} we introduce the abstract setting and the assumptions which we will use.
Here we illustrate the applicability of our setting by several examples.
 In Lemma \ref{lem:eu} we  state and prove
a basic existence and uniqueness result for the solution of \eqref{varcons0} and specify  the time-regularity of the solution in Remark \ref{rem:reg}.
 Our main abstract approximation result estimating the difference  $U-U_n$ in  $L^p(\Omega;C^\gamma([0,T];H_0))$ is contained in Theorem \ref{thm:main} while in $L^p(\Omega;C([0,T];H_0))$ in Corollary \ref{cor:main}.
In Section \ref{app} we apply Theorem \ref{thm:main} and Corollary \ref{cor:main} to two typical space discretization schemes.
In particular,
in Subsection \ref{sg_one} we consider the general fractional-order equation \eqref{SDE-intro} (with time-independent coefficients, for simplicity) and its spectral Galerkin approximation and apply Theorem \ref{thm:main} and Corollary \ref{cor:main} to obtain rates of convergence depending on the parameters in the equation. In Examples \ref{ex:sg1}, \ref{ex:sg2}, and \ref{ex:sg3} we state the results in some simplified settings for the stochastic heat equation, the fractional stochastic heat and wave equations, respectively.
In Subsection \ref{sg_two} we consider a fractional stochastic wave equation and its finite element approximation and apply Theorem \ref{thm:main} and Corollary \ref{cor:main} again to obtain rates of convergence. Here, in Remark \ref{rem:she}, we point out in which way the stochastic heat equation fits in our abstract framework and we also show that using the setup of the paper one may remove some unnecessary smoothness
  assumption on $G$ which was present in \cite[Proposition 4.2]{sonja3}. Finally, in Section \ref{sec:ne}, we present some numerical experiments for a fractional stochastic wave equation to verify the theoretical rates obtained in Subsections \ref{sg_one} and \ref{sg_two}. In particular, in Subsection \ref{sec:numerical:MLEI}, we present some numerical results for the spectral Galerkin approximation and space-time white noise, while in Subsection \ref{sec:numerical:LCQM}, we describe some experiments for the finite element method and trace-class noise.

\subsection*{Notation} We denote by $\RR_0^+$ the set $\{t\in\RR:t\geq 0\}$. For Banach spaces $V$ and $W$ we denote by $\mathcal{L}(V,W)$ the space of bounded linear operators from $V$ into $W$  endowed with the  norm
$$
\|A\|_{\mathcal{L}(V,W)} = \inf\{ C \geq 0 : \|Av\|_W \leq C \|v\|_V \mbox{ for all } v \in V \}, ~\text{for }A\in \mathcal{L}(V,W).
$$
If $V=W$, we write $\mathcal{L}(V)$ for $\mathcal{L}(V,W)$ and we denote the norm by $\|\cdot\|_{\mathcal{L}(V)}$. For an operator valued function $S:[a,b]\to \mathcal{L}(V,W)$ we use the notation 
$$
\dot{S}(t)v:=\frac{d}{dt}\big(t\mapsto S(t)v\big)(t),\,v\in V,
$$ whenever $t\mapsto S(t)v$ is differentiable at $t$.
Furthermore, for two real separable Hilbert spaces $V$ and $W$,  we denote by $L_{HS}(V,W)$ the space of Hilbert-Schmidt operators from $V$ to $W$ equipped with the norm
$$
\|T\|_{\fahimm{L_{HS}(V,W)}}^2=\sum\limits_{n=1}^{\infty}\|Te_n\|_W^2,~\text{for }T\in L_{HS}(V,W),
$$
for an orthonormal basis $(e_n)\subset V$.
Let $H$ be a real, separable, infinite-dimensional Hilbert space and 
let $A:D(A)\subseteq H \to H$, be an unbounded, self-adjoint, and positive definite operator with compact inverse. For $\xi\in \mathbb{R}$ one defines the fractional power $A^\xi$ of $A$ via the standard spectral functional calculus of $A$. For $\xi\geq 0$ we equip $D(A^\xi)$, where $D(A^\xi)$ denotes the domain of $A^\xi$, with the norm $\|x\|_{D(A^\xi)}:=\|A^\xi x\|_H$, $x\in A^\xi$.
\fahim{For $\delta\geq 0$,} let $H_{-\delta}^A$, denote the completion of $H$ with respect to the norm 
$
\|x\|_{\fahim{-\delta}}:=\|A^{-\delta}x\|_{H}, 
$
\fahim{ $x\in H_{-\delta}^A$.}
Let $E$ be a Banach space. We  denote by $L^p([0,T];E)$, $1\leq p<\infty$, the space of all measurable functions $f:[0,T]\to E$ being $L^p$ integrable equipped the with the standard norm
$$
\|f\|_{L^p([0,T];E)}:=\Big(\int_0^T \|\fahimm{f(t)} \|_E^p~dt\Big)^{1/p}.
$$
Moreover if $p=\infty$, then $L^\infty([0,T];E)$ denotes  the space of all measurable functions  $f$ from $[0,T]$ to $E$ being essential bounded in $E$
 on $[0,T]$ equipped  with the norm
$$
\|f\|_{L^{\infty}([0,T];E)}=\esssup_{t\in[0,T]}\|f(t) \|_E.
$$
We denote by $C([0,T];E)$ the space of continuous functions $f:[0,T]\to E$ endowed with the usual supremum norm.
Let $C^\gamma([0,T];E)$, $0 <\gamma<1$, denote the space of functions $f:[0,T]\to E$ such that the seminorm
$$
\|f\|_{C^\gamma([0,T];E)}:=\sup_{\stackrel{t,s \in [0,T]}{t\neq s}}\frac{\|f(t)-f(s)\|_E}{|t-s|^\gamma}<\infty.
$$
Let $(\Omega,\mathcal{F},\mathbb{P})$ be a probability space and let
$L^p(\Omega;E)$ denote the space of random variables $X\colon
(\Omega,\mathcal{F})\rightarrow (E,\mathcal{B}(E))$; that is, $\mathcal{F}/\mathcal{B}(E)$-measurable mappings $X:\Omega\to E$, where $\mathcal{B}(E)$ denotes
the Borel $\sigma$-algebra of $E$, such
that
\begin{equation*}
\|X\|_{L^p(\Omega;E)}^p=\mathbb{E}\big(\|X\|_{E}^p\big)
=\int\limits_{\Omega}\|X(\omega)\|_{E}^p\,d \mathbb{P}(\omega)<\infty.
\end{equation*}

\section{The  abstract approximation result}\label{AAR}

Let $H$ 
and $H_0$ be two real separable Hilbert spaces. Let  $W_H$ be a
$H$-cylindrical Wiener process on a complete, filtered probability space $(\Omega,\CF,(\CF_t)_{t\ge0},\PP)$ with respect to the filtration $(\CF_t)_{t\ge0}$, the latter satisfying the usual conditions.
We consider the following integral equation
\begin{equation}\label{varcons}
U(t)=X^0(t)+\int_{0}^{t}S^2(t-s)F(s,U(s))\,ds+\int_{0}^{t}S^1(t-s)G(s,U(s))\,dW_H(s). 
\end{equation}

To specify the assumptions on the coefficients $F$ and $G$, let us fix two real separable Hilbert spaces $H_1$ and $H_2$.  Later on, we will see in the examples that these spaces will be interpolation spaces associated with a linear operator.
\begin{ass}\label{hypogeneral}
We assume that
\begin{enumerate}
\item the mapping $G: [0,T]\times H_0\to L_{HS}(H,H_1)$ is Lipschitz continuous and of linear growth in the second variable uniformly in $[0,T]$; that is, there is a constant $C_G>0$ such that
$$
\|G(t,u)-G(s,v)\|_{ L_{HS}(H,H_1)}\leq C_G(|t-s|+\|u-v\|_{H_0}),\,u,v\in H_0, \,s,t\in [0,T];
$$
and
$$
\|G(t,u)\|_{ L_{HS}(H,H_1)}\leq C_G(1+\|u\|_{H_0}),\,u\in H_0, \,t\in [0,T];
$$
\item the mapping $F: [0,T]\times H_0\to H_2$ is Lipschitz continuous and of linear growth in the second variable uniformly in $[0,T]$; that is
there is a constant $C_F>0$ such that
$$
\|F(t,u)-F(s,v)\|_{H_2}\leq C_F(|t-s|+\|u-v\|_{H_0}),\,u,v\in H_0, \,s,t\in [0,T];
$$
and
$$
\|F(t,u)\|_{ H_2}\leq C_F(1+\|u\|_{H_0}),\,u\in H_0, \,t\in [0,T];
$$

\item the $H_0$-valued process $\{X^0(t)\}_{t\in [0,T]}$ is predictable and, for some $p>2$, $$X^0\in L^p(\Omega; L^\infty([0,T];H_0)).$$
\end{enumerate}
\end{ass}
We note that the $t$-dependence of $F$ and $G$ may be weakened considerably and they may also be stochastic.
Concerning the families $S^i$ we suppose the following.
\begin{ass}\label{ass:semigroup}
We assume that the linear operator families $S^i$ are strongly continuously differentiable as operators from $H_i$ to $H_0$ on $(0,T)$, $i=1,2$. Furthermore, we assume that
\begin{enumerate}
  \item there exists a function $s_1\in L^ 1 ([0,T];\RR^+_0)$ and a constant $0<\gamma_1<\frac 12 $ such that
\begin{align}\label{ass:psi}
  t^{\gamma_1} \n \dot S^1(t) x\n_{H_0} + t^{\gamma_1-1} \n  S^1(t) x\n_{H_0}  &
\leq s_1(t)\n x\n_{H_1}, \, \textrm{for all } x\in H_1,\,t\in (0,T);
\end{align}
  \item there exists a function $s_2\in L^ 1 ([0,T];\RR^+_0)$ and a constant $0<\gamma_2<1$ such that
\begin{align}\label{ass:psi1}
  t^{\gamma_2} \n \dot S^2(t) x\n_{H_0} + t^{\gamma_2-1} \n  S^2(t) x\n_{H_0}  &
\leq s_2(t)\n x\n_{H_2}, \, \textrm{for all } x\in H_2,\,t\in (0,T).
\end{align}
\end{enumerate}
\end{ass}
We consider an approximation of \eqref{varcons}  given by the following integral equation
\begin{equation}\label{varconsn}
U_n(t)=X_n^0(t)+\int_{0}^{t}S_n^2(t-s)F(s,U_n(s))\,ds+\int_{0}^{t}S_n^1(t-s)G(s,U_n(s))\,dW_H(s), 
\end{equation}
where  $\{X_n^0(t)\}_{t\in [0,T]}$ is $H_0$-predictable and $X^0_n\in  L^p(\Omega; L^\infty([0,T];H_0))$. Concerning the approximation we make the following assumptions.
\begin{ass}\label{ass:semigroup:app} Let $\gamma_1$ and $\gamma_2$ \fahim{be} as in Assumption \ref{ass:semigroup}.
For bounded functions $r_i:\NN\to \RR^ +$, $i=1,2$
consider
$$
\Psi^{i}_\h(t):= \frac 1 {r_i(n)}
 \lk[ S^i(t)- S^i_\h(t)\rk],\,i=1,2.
$$
We assume that the linear operator families $S^i_n$ are strongly continuously differentiable as operators from $H_i$ to $H_0$  on $(0,T)$, $i=1,2$. Furthermore, we assume that
\begin{enumerate}
  \item there exists a function $h_1\in L ^1 ([0,T];\RR ^+_0)$  such that  for all $\h\in\NN$ we have
\begin{align}\label{ass:psi2}
  t^{\gamma_1} \n \dot \Psi^1_\h(t) x\n_{H_0} + t^{\gamma_1-1} \n  \Psi^1_\h(t) x\n_{H_0}  &
\leq h_1(t)\n x\n_{H_1}, \, \textrm{for all } x\in H_1,\,t\in (0,T);
\end{align}
  \item there exists a function $h_2\in L ^1 ([0,T];\RR ^+_0)$
   such that for all $\h\in\NN$ we have
\begin{align}\label{ass:psi22}
  t^{\gamma_2} \n \dot \Psi^2_\h(t) x\n_{H_0} +  t^{\gamma_2-1} \n  \Psi^2_\h(t) x\n_{H_0}  &
\leq h_2(t)\n x\n_{H_2}, \, \textrm{for all } x\in H_2,\,t\in (0,T).
\end{align}
\end{enumerate}
\end{ass}
Note that Assumptions \ref{ass:semigroup} and \ref{ass:semigroup:app} imply that for some $C>0$, for all $n\in \N$,
  \begin{align}\label{eq:psin}
  t^{\gamma_1} \n \dot S_n^1(t) x\n_{H_0} + t^{\gamma_1-1} \n  S_n^1(t) x\n_{H_0}  &
\leq (s_1(t)+Ch_1(t))\n x\n_{H_1},
\end{align}
 for all  $x\in H_1$, $t\in (0,T)$, and
\begin{align}\label{eq:psinn}
  t^{\gamma_2} \n \dot S_n^2(t) x\n_{H_0} + t^{\gamma_2-1} \n  S_n^2(t) x\n_{H_0}  &
\leq (s_2(t)+Ch_2(t))\n x\n_{H_2},
\end{align}
 for all  $x\in H_2$, $t\in (0,T)$.\\

 \begin{example}[Stochastic heat equation]\label{ex:main} In order to illustrate  a typical situation of our basic assumptions, we first consider  the familiar setting of the heat equation.
 Let $A:D(A)\subset H_0\to H_0$ be an unbounded, densely defined, self-adjoint, positive definite operator with compact inverse. Let $\lambda_n$ denote the eigenvalues of $A$, arranged in a non-decreasing order, with corresponding orthonormal eigenbasis $(e_n)\subset H_0$. A typical example is when $H_0=L^2(\mathcal{D})$, where $\mathcal{D}\subset \R^d$ is a bounded domain with smooth or convex polygonal boundary, and $A=-\Delta$ with Dirichlet zero boundary conditions.
 In this case, we have  $S^1(t)=S^2(t)=e^{-tA}$.
 Typically the spaces $H_1$ and $H_2$ are related to the fractional powers of $A$ but for simplicity we take  $H_1=H_2=H=H_0$. In this case, Assumptions  \ref{hypogeneral} become standard global Lipschitz assumptions on the coefficients in the equation. Furthermore, due to the analyticity of the semigroup $S(t):=e^{-tA}$ one has the well-known smoothing properties
 $$
 \|A^\xi S(t)x\|_{H_0}\leq Mt^{-\xi}\|x\|_{H_0},\quad  \|A^\xi \dot S(t)x\|_{H_0}\leq Mt^{-\xi-1}\|x\|_{H_0},
 $$
 for $t>0$ and $\xi\geq 0$. Then Assumption \ref{ass:semigroup}-(1) is satisfied for any $0<\gamma_1<\frac12$ for $s_1(t)=Mt^{\gamma_1-1}$ and  Assumption \ref{ass:semigroup}-(2) is satisfied for any $0<\gamma_2<1$ for $s_2(t)=Mt^{\gamma_2-1}$. The simplest example of an approximation procedure that we have in mind is the spectral Galerkin method. We define a family of finite-dimensional
subspaces \fahim{$\{H^n:n\in\NN\}$} of $H_0$ by
$$\fahim{H^n}=\text{span}\{e_1,e_2,\dots,e_n\}$$ and define the orthogonal projection
\begin{equation} \label{PN}
 \mathcal{P}_n\colon H_0\to \fahim{H^n},\quad
  \mathcal{P}_n x=\sum_{k=1}^n ( x,e_k)_{H_0}
   e_k,\quad x\in H_0,
\end{equation}
where $(\cdot\,,\cdot)_{H_0}$ denotes the inner product of $H_0$. It is easy to see that
\begin{equation}
 \|A^{-\nu}(I-\mathcal{P}_n)\|_{\mathcal{L}(H_0)}= \|(I-\mathcal{P}_n)A^{-\nu}\|_{\mathcal{L}(H_0)}
  =\sup_{k\geq n+1}\lambda_k^{-\nu}
  =\lambda_{n+1}^{-\nu},\quad \nu \geq 0.
  \label{PN:est}
\end{equation}
The approximating operators become $$S^1_n(t)=S^2_n(t)=\mathcal{P}_nS(t)=S(t)\mathcal{P}_n:=S_n(t).$$
Using  eigenfunctions and eigenvalues of $A$ we can write
$$
S_n(t)x=\sum_{k=1}^ne^{-\lambda_k t}(x,e_k)_{H_0}e_k.
$$
For $\nu\geq 0$, we set $r_1(n)=r_2(n)=\lambda_{n+1}^{-\nu}$. We then have
$$
\Psi^1_n(t)=\Psi^2_n(t)=\lambda_{n+1}^{\nu}(S(t)-S_n(t))=\lambda_{n+1}^{\nu}(I-\mathcal{P}_n)S(t)
$$
with
$$
\|\Psi^i_n(t)x\|_{H_0}\leq C\|A^\nu S(t)x\|_{H_0}\leq C t^{-\nu}\|x\|_{H_0},\,i=1,2\text{ and }n\in \N,
$$
and
$$
\|\dot{\Psi}^i_n(t)x\|_{H_0}\leq C\|A^\nu \dot{S}(t)x\|_{H_0}\leq C t^{-\nu-1}\|x\|_{H_0},\,i=1,2\text{ and }n\in \N.
$$
Thus, Assumption \ref{ass:semigroup:app} is satisfied with $r_1(n)=\lambda_{n+1}^{-\nu_1}$ with $\nu_1<\gamma_1$ and $h_1(t)=Ct^{\gamma_1-1-\nu_1}$ and $r_2(n)=\lambda_{n+1}^{-\nu_2}$ with $\nu_2<\gamma_2$ and $h_2(t)=Ct^{\gamma_2-1-\nu_2}$.
 \end{example}
 \begin{example}[Fractional stochastic heat equation]\label{ex:main1} Here we consider the fractional stochastic heat equation \eqref{eq:shebase} with mild solution given by \eqref{eq:UUUU} with parameters $\alpha\in (0,1)$, $\beta=\kappa=1$, $u_0=U_0$ and $u_1=0$, where the operator family $S^{\alpha,\beta}$ is defined by \eqref{eq:shatab} via its Laplace transform. We use the setting of the previous example for $A$, $F$ and $G$; that is, consider the global Lipschitz case. In this case we have $S^{1}(t)=S^2(t)=S^{\alpha,1}(t)$ with smoothing properties specified in  Lemma \ref{lem:sm}. In particular, we have
  $$
 \|A^\xi S^i(t)x\|_{H_0}\leq Mt^{-\alpha \xi}\|x\|_{H_0},\quad  \|A^\xi \dot S^i(t)x\|_{H_0}\leq Mt^{-\alpha\xi-1}\|x\|_{H_0},\,i=1,2,
 $$
 for $\xi\in [0,1]$ and $t>0$. Then, Assumption \ref{ass:semigroup}-(1) is satisfied for any $0<\gamma_1<\frac12$ for $s_1(t)=Mt^{\gamma_1-1}$ and  Assumption \ref{ass:semigroup}-(2) is satisfied for any $0<\gamma_2<1$ for $s_2(t)=Mt^{\gamma_2-1}$. The approximating operators in this case become
 $$
 S^1_n(t)=S^2_n(t)=\mathcal{P}_nS^{\alpha,1}(t)=S^{\alpha,1}(t)\mathcal{P}_n:=S^{\alpha,1}_n(t).
 $$
  For $\nu\in [0,1]$, we set $r_1(n)=r_2(n)=\lambda_{n+1}^{-\nu}$. We then have
$$
\Psi^1_n(t)=\Psi^2_n(t)=\lambda_{n+1}^{\nu}(S^{\alpha,1}(t)-S^{\alpha,1}_n(t))=\lambda_{n+1}^{\nu}(I-\mathcal{P}_n)S^{\alpha,1}(t)
$$
with
$$
\|\Psi^i_n(t)x\|_{H_0}\leq C\|A^\nu S^{\alpha,1}(t)x\|_{H_0}\leq C t^{-\alpha\nu}\|x\|_{H_0},\,i=1,2\text{ and }n\in \N,
$$
and
$$
\|\dot{\Psi}^i_n(t)x\|_{H_0}\leq C\|A^\nu \dot{S}^{\alpha,1}(t)x\|_{H_0}\leq C t^{-\alpha\nu-1}\|x\|_{H_0},\,i=1,2\text{ and }n\in \N.
$$
Thus, Assumption \ref{ass:semigroup:app} is satisfied with $r_1(n)=\lambda_{n+1}^{-\nu_1}$ with $\nu_1<\frac{\gamma_1}{\alpha}$ and $h_1(t)=Ct^{\gamma_1-1-\alpha\nu_1}$ and $r_2(n)=\lambda_{n+1}^{-\nu_2}$ with $\nu_2<\frac{\gamma_2}{\alpha}$ and $h_2(t)=Ct^{\gamma_2-1-\alpha\nu_2}$. It is important to note that the additional restriction $\max(\nu_1,\nu_2)\leq 1$ applies as $S^{\alpha,1}$ has only the limited smoothing properties shown in Lemma \ref{lem:sm}. This implies that the rates improve as $\alpha$ decreases to $\gamma_i$, however, remain constant once we have $\alpha\leq \gamma_i$.

 \end{example}

 \begin{example}[Fractional stochastic wave equation]\label{ex:main2} Here we consider the fractional stochastic wave equation \eqref{swebase} with mild solution given by \eqref{eq:UUUU} with parameters $\alpha\in (1,2)$, $\beta=\kappa=1$, $u_0=U_0$ and $u_1=U_1$,  where the operator family $S^{\alpha,\beta}$ is defined by \eqref{eq:shatab} via its Laplace transform. We use the setting of the previous example for $A$, $F$ and $G$; that is, consider the global Lipschitz case. \fahim{In this case we have} $S^{1}(t)=S^2(t)=S^{\alpha,1}(t)$ with smoothing properties specified by Lemma \ref{lem:sm}
  $$
 \|A^\xi S^i(t)x\|_{H_0}\leq Mt^{-\alpha \xi}\|x\|_{H_0},\quad  \|A^\xi \dot S^i(t)x\|_{H_0}\leq Mt^{-\alpha\xi-1}\|x\|_{H_0},\quad \fahim{i=1,2,}
 $$
 for $\xi\in [0,1]$ and $t>0$. Then, Assumption \ref{ass:semigroup} (1) is satisfied for any $0<\gamma_1<\frac12$ for $s_1(t)=Mt^{\gamma_1-1}$ and  Assumption \ref{ass:semigroup} (2) is satisfied for any $0<\gamma_2<1$ for $s_2(t)=Mt^{\gamma_2-1}$. The approximating operators in this case become $$S^1_n(t)=S^2_n(t)=\mathcal{P}_nS^{\alpha,1}(t)=S^{\alpha,1}(t)\mathcal{P}_n:=S^{\alpha,1}_n(t).$$ For $\nu\in [0,1]$, we set $r_1(n)=r_2(n)=\lambda_{n+1}^{-\nu}$. We then have
$$
\Psi^1_n(t)=\Psi^2_n(t)=\lambda_{n+1}^{\nu}(S^{\alpha,1}(t)-S^{\alpha,1}_n(t))=\lambda_{n+1}^{\nu}(I-\mathcal{P}_n)S^{\alpha,1}(t)
$$
with
$$
\|\Psi^i_n(t)x\|_{H_0}\leq C\|A^\nu S^{\alpha,1}(t)x\|_{H_0}\leq C t^{-\alpha\nu}\|x\|_{H_0},\,i=1,2\text{ and }n\in \N,
$$
and
$$
\|\dot{\Psi}^i_n(t)x\|_{H_0}\leq C\|A^\nu \dot{S}^{\alpha,1}(t)x\|_{H_0}\leq C t^{-\alpha\nu-1}\|x\|_{H_0},\,i=1,2\text{ and }n\in \N.
$$
Thus, Assumption \ref{ass:semigroup:app} is satisfied with $r_1(n)=\lambda_{n+1}^{-\nu_1}$ with $\nu_1<\frac{\gamma_1}{\alpha}$ and $h_1(t)=Ct^{\gamma_1-1-\alpha\nu_1}$ and $r_2(n)=\lambda_{n+1}^{-\nu_2}$ with $\nu_2<\frac{\gamma_2}{\alpha}$ and $h_2(t)=Ct^{\gamma_2-1-\alpha\nu_2}$. We see here that the rate deteriorates with increasing $\alpha$.

 \end{example}
We will often make use of the following results on the H\"older regularity of deterministic and stochastic convolutions.
\begin{lemma}\label{prop:detHolder}
Let $Y_1$ and $Y_2$ be real separable Hilbert spaces. Let $T>0$ and suppose that $\Phi\in L^ {p}(\Omega;L^{\infty}([0,T];Y_1))$ for some $p\in
[1,\infty)$.
Let $\Psi:[0,T]\rightarrow \calL(Y_1,Y_2)$ be a mapping  such that the mapping $t\mapsto\Psi(t) x$ is continuously
differentiable on $(0,T)$ for all $x\in
 Y_1$. Suppose, moreover, that there exists a function $g\in L ^1 ([0,T];\RR ^+_0)$ and a
constant $\theta\in (0,1)$ such that for all $t\in (0,T)$ we have
\begin{align*}
t^\theta \n \dot \Psi(t) x\n_{Y_2} + \theta t^{\theta-1} \n \Psi(t)
x\n_{Y_2}  & \leq g(t)\n x\n_{Y_1}, \quad \textrm{for all } x\in Y_1.
\end{align*}
Then,
\begin{enumerate}[label=(\alph*)]
\item  the convolution
$$(\Psi\ast\Phi):t\mapsto
\int_{0}^{t}\Psi(t-s)\Phi(s)\,ds
$$ is well-defined almost surely;

\item there is $\bar{C}>0$, depending only on $\theta$, such that
\begin{align*} & \Big\n t\mapsto (\Psi \ast\Phi)(t)
\Big\n_{L^{p}(\Omega;C^{1-\theta}([0,T];Y_2))} \\
& \qquad \qquad  \leq  \bar{C} \n g \n_{L ^1 ([0,T];\RR ^+_0)} \lk\| \Phi\rk\|_{L^ {p}(\Omega;L^{\infty}([0,T];Y_1))};
\end{align*}
\item there is $\tilde{C}>0$, depending only on $\theta$ and $T$, such that
\begin{align*} & \Big\n t\mapsto (\Psi \ast\Phi)(t)
\Big\n_{L^{p}(\Omega;C([0,T];Y_2))} \\
& \qquad \qquad  \leq  \tilde{C} \n g \n_{L ^1 ([0,T];\RR ^+_0)} \lk\| \Phi\rk\|_{L^ {p}(\Omega;L^{\infty}([0,T];Y_1))}.
\end{align*}

\end{enumerate}
\end{lemma}
\begin{proof}
Note first that, almost surely, the mapping $s\mapsto \Psi(t-s)\Phi(s)\in L^1([0,T]; Y_2)$ and hence $\Psi\ast\Phi$ is well defined almost surely.
In \cite[Proposition 3.6]{sonja1} it is shown, that under the assumptions of the theorem, almost surely, there is $\bar{C}>0$, depending only on $\theta$, such that
$$
\Big\n t\mapsto (\Psi \ast\Phi)(t)
\Big\n_{C^{1-\theta}([0,T];Y_2)}\leq \bar{C}\n g \n_{L ^1 ([0,T];\RR ^+_0)} \lk\| \Phi\rk\|_{L^{\infty}([0,T];Y_1)}.
$$
The estimate in (b) follows by taking the $p$th power and expected value of both sides of the inequality. Finally the estimate in (c) follows from the estimate in (b) by noting that $(\Psi\ast\Phi)(0)=0$.
\end{proof}

\begin{lemma}\label{lem:stochConv}
Let $Y_1$ and $Y_2$ be real separable Hilbert spaces.
Let $T>0$ and suppose that the process $\Phi:\Omega\times [0,T]\to L_{HS}(H,Y_1)$ is predictable and that $\Phi\in L^{p'}(\Omega;L^{\infty}([0,T];L_{HS}(H,Y_1)))$ for some $p'\in (2,\infty]$.
%
Let $\Psi:[0,T]\rightarrow \calL(Y_1,Y_2)$ be a mapping  such that the mapping $t\mapsto \Psi(t) x$ is
continuously differentiable on $(0,T)$ for all $x\in
Y_1$. Suppose, moreover, that there exists a function $g\in L ^1 ([0,T];\RR ^+_0)$ and a
constant $\theta\in (0,1)$ such that for all $t\in (0,T)$,
\begin{align}\label{ass:psie}
  t^\theta \n \dot\Psi(t) x\n_{Y_2} + \theta \, t^{\theta-1} \n  \Psi(t) x\n_{Y_2}  &
\leq g(t)\n x\n_{Y_1}, \quad \textrm{for all } x\in Y_1.
\end{align}
Then,
\begin{enumerate}[label=(\alph*)]
\item  the stochastic convolution process
$$(\Psi\diamond\Phi):t\mapsto
\int_{0}^{t}\Psi(t-s)\Phi(s)\,dW_H(s)
$$ is well-defined;

\item for any  $\gamma \in (0,\frac12-\theta-\frac 1{p'})$ with $\theta<\frac 12 -\frac 1{p'}$ there exists a modification of  $\Psi\diamond\Phi$ such that
\begin{align*}
& \Big\n t\mapsto (\Psi \diamond\Phi)(t)
\Big\n_{L^{p'}(\Omega;C^{\gamma}([0,T];Y_2))} \\
& \qquad \qquad  \leq  \bar{C} \n g \n_{L ^1 ([0,T];\RR ^+_0)} \lk\| \Phi\rk\|_{L^{p'}(\Omega;L^{\infty}(0,T;L_{HS}(H,Y_1)))},
\end{align*}
where $\bar{C}$ only depends on $\gamma,\theta,p'$ and $T$;
\item the modification of  $\Psi\diamond\Phi$ from (b) also satisfies
\begin{align*}
& \Big\n t\mapsto (\Psi \diamond\Phi)(t)
\Big\n_{L^{p'}(\Omega;C([0,T];Y_2))} \\
& \qquad \qquad  \leq  \tilde{C} \n g \n_{L ^1 ([0,T];\RR ^+_0)} \lk\| \Phi\rk\|_{L^{p'}(\Omega;L^{\infty}(0,T;L_{HS}(H,Y_1)))},
\end{align*}
where $\tilde{C}$ only depends on $\gamma,\theta,p'$ and $T$.
\end{enumerate}
\end{lemma}
\begin{proof}
Let $0<\eta<\frac12$. Then
\begin{equation}\label{eq:psipsip}
\sup_{0\leq t\leq T}\|s\mapsto (t-s)^{-\eta}\Phi(s)\|_{L^{p'}(\Omega;L^{2}([0,T];L_{HS}(H,Y_1)))}\leq C_{\eta,T} \|\Phi\|_{L^{p'}(\Omega;L^{\infty}([0,T];L_{HS}(H,Y_1)))}.
\end{equation}
As $Y_1$ is a Hilbert space, the statement in (a) follows from \cite[Lemma 3.2]{sonja1} by noting that in this case $L^{2}([0,T];L_{HS}(H,Y_1))\simeq \gamma([0,T];H,Y_1)$, where the latter denotes the space of gamma radonifying operators \fahim{from} $L^2([0,T];H)\to Y_1$, see \cite[Section 2.2]{sonja1} for further details. Let $\gamma \in (0,\frac12-\theta-\frac 1{p'})$. Then, there is $0<\eta<\frac12$ such that $\gamma \in (0,\eta-\theta-\frac 1{p'})$. Then,
by \cite[Corollary 3.4]{sonja1}, there exists a modification of  $\Psi\diamond\Phi$ and a constant $C$ depending on $\gamma, \eta, p' $ such that
\begin{align*}
& \Big\n t\mapsto (\Psi \diamond\Phi)(t)
\Big\n_{L^{p'}(\Omega;C^{\gamma}([0,T];Y_2))} \\
& \qquad \qquad  \leq  C \n g \n_{L ^1 ([0,T];\RR ^+_0)} \sup_{0\leq t\leq T}\|s\mapsto (t-s)^{-\eta}\Phi(s)\|_{L^{p'}(\Omega;L^{2}([0,T];L_{HS}(H,Y_1)))}\\
&  \qquad \qquad  \leq  \bar{C}  \n g \n_{L ^1 ([0,T];\RR ^+_0)}  \|\Phi\|_{L^{p'}(\Omega;L^{\infty}([0,T];L_{HS}(H,Y_1)))},
\end{align*}
where $\bar{C}$ depends on $\gamma,p', T$ and $\eta$, with the latter ultimately depending on $\gamma,\theta$ and $p'$. We used \cite[Corollary 3.4]{sonja1} in the first inequality and \eqref{eq:psipsip} in the second. Finally the estimate in (c) follows from the estimate in (b) by noting that $(\Psi\diamond\Phi)(0)=0$.
\end{proof}

Next we state a basic existence and uniqueness result.
\begin{lemma}\label{lem:eu}
Let $p>2$ and let Assumption \ref{hypogeneral}, Assumption \ref{ass:semigroup}, and Assumption \ref{ass:semigroup:app} be satisfied with  $0<\gamma_1<\frac12 -\frac{1}{p}$ and $0<\gamma_2<1 $.
Then, equations \eqref{varcons} and \eqref{varconsn} have  unique $H_0$-predictable solutions $U$, respectively $U_n$, in $L^p(\Omega; L^\infty([0,T];H_0))$ with
\begin{equation}\label{eq:stability}
\begin{aligned}
&\|U\|_{L^p(\Omega; L^\infty([0,T];H_0))}\le C(1+\|X^0\|_{L^p(\Omega; L^\infty([0,T];H_0))}); \\
&\|U_n\|_{L^p(\Omega; L^\infty([0,T];H_0))}\le C(1+\|X^0_n\|_{L^p(\Omega; L^\infty([0,T];H_0))}) ,
\end{aligned}
\end{equation}
for some $C>0$ depending on $\gamma_1,\gamma_2,p$ and $T$.
\end{lemma}
\begin{proof}
The proof is fairly standard as it uses Banach's fixed point theorem and therefore we only sketch a proof. Let $T>0$ and set
$$
X_T:=\{U\in L^p(\Omega; L^\infty([0,T];H_0)):\,U\text{ is } H_0\text{-predictable}\}.
$$
For $\lambda>0$, later to be chosen appropriately,  we endow $X_T$ with the norm
$$
\|U\|^p_\lambda:=\EE \esssup_{t\in [0,T]} e^{-\lambda p t}\|U(t)\|_{H_0}^p.
$$
Note, the  latter definition is equivalent to the natural norm of $L^p(\Omega; L^\infty([0,T];H_0))$. For $ U\in X_T$ define the fixed point map
\begin{multline*}
\mathcal{L}(U)(t)=X^0(t)+\int_{0}^{t}S^2(t-s)F(s,U(s))\,ds+\int_{0}^{t}S^1(t-s)G(s,U(s))\,dW_H(s)
\\
=:X^0(t)+\mathcal{L}_2(U)(t)+\mathcal{L}_1(U)(t).
\end{multline*}
For $\gamma<\min(\frac 12 -\frac 1p-\gamma_1,1 -\gamma_2)$  and  for $i=1,2$, we have
\begin{equation}\label{eq:smoothing}
\begin{aligned}
&\qquad \left\|\mathcal{L}_i(U)(t)\right\|_{L^p(\Omega; L^\infty([0,T];H_0))}\\
&\leq C_T \left\|\mathcal{L}_i(U)(t)\right\|_{L^{p}(\Omega;C^{\gamma}([0,T];H_0))}\leq C \|s_i\|_{L ^1 ([0,T];\RR ^+_0)}\left(1+\|U\|_{L^p(\Omega; L^\infty([0,T];H_0)}\right).
\end{aligned}
\end{equation}
In the first inequality above we used the fact that $\mathcal{L}_i(U)(0)=0$. In the second inequality, for $i=1$, we used Lemma \ref{prop:detHolder} with $Y_1=H_2$, $Y_2=H_0$ and the linear growth of $F$ while, for $i=2$,
we used Lemma \ref{lem:stochConv} with $Y_1=H_1$, $Y_2=H_0$ and the linear growth of $G$.
Hence, the operator $\mathcal{L}$ is a mapping $X_T\to X_T$. Next we show that $\mathcal{L}$ is a contraction for $\lambda$ large enough. Indeed, if $U,V\in X_T$, then
\begin{align*}
&\left\|\mathcal{L}_2(U)(t)-\mathcal{L}_2(V)(t)\right\|^p_{\lambda} \\
&=
 \EE \esssup_{t\in [0,T]} \left\|\int_0^te^{-\lambda (t-s)}S^2(t-s)e^{-\lambda s}\left(F(s,U(s))-F(s,V(s))\right)\,ds\right\|_{H_0}^p \\
&  \leq C_T \left\|t\mapsto \int_0^te^{-\lambda (t-s)}S^2(t-s)e^{-\lambda s}\left(F(s,U(s))-F(s,V(s))\right)\,ds\right\|_{L^{p}(\Omega;C^{\gamma}([0,T];H_0))}^p\\
&\leq C_{p,T} \left(\int_0^ T (2+\lambda t) e^{-\lambda t}s_2(t)\, dt\right)^p\|U-V\|_{\lambda}^p,
\end{align*}
see the estimate of $I_1$ in the proof of Theorem \ref{thm:main} for more details. We similarly have that
$$
\left\|\mathcal{L}_1(U)(t)-\mathcal{L}_1(V)(t)\right\|^p_{\lambda} \leq C_{p,T} \left(\int_0^ T (2+\lambda t) e^{-\lambda t}s_1(t)\, dt\right)^p\|U-V\|_{\lambda}^p,
$$
see the estimate of $I_3$ in the proof of Theorem \ref{thm:main} for more details. Thus, by choosing $\lambda>0$ large enough,
we conclude that there is a constant $M\in(0,1)$ such that $\left\|\mathcal{L}(U)(t)-\mathcal{L}(V)(t)\right\|_{\lambda}\leq M\|U-V\|_\lambda$ when $\mathcal{L}:X_T\to X_T$ is a contraction. Therefore, by Banach's fixed point theorem, there exists a unique $U\in X_T$ with $U=\mathcal{L}(U)$. A similar argument shows, by defining the corresponding mapping $\mathcal{L}_n$ in an obvious way, the existence and uniqueness of $U_n$ with the terms $(2+\lambda t) e^{-\lambda t}s_i(t)$ above replaced by $(2+\lambda t) e^{-\lambda t}(s_i(t)+Ch_i(t))$ (c.f., \eqref{eq:psin} and \eqref{eq:psinn}) yielding a uniform in $n$ contraction constant $0<M'<1$ of $\mathcal{L}_n$. The estimate \eqref{eq:stability} follows from the following simple estimate
$$
\|U\|_{\lambda}=\|\mathcal{L}(U)\|_\lambda \leq \|\mathcal{L}(U)-\mathcal{L}(0)\|_\lambda+\|\mathcal{L}(0)\|_\lambda \leq M\|U\|_{\lambda}+\|X^0\|_\lambda+C_{p,T,\lambda},
$$
where $\lambda$ is large enough so that $0<M<1$ and similarly for $U_n$.
\end{proof}

\begin{remark}[Regularity]\label{rem:reg} Observe that under the assumptions of Lemma \ref{lem:eu}, if also $X_0,X_n^0\in L^p(\Omega;C^\gamma([0,T];H_0))$ holds for $\gamma<\min(\frac 12 -\frac 1p-\gamma_1,1 -\gamma_2)$, then
\begin{equation}\label{eq:stability1}
\begin{aligned}
&\|U\|_{ L^p(\Omega;C^\gamma([0,T];H_0))}\le C(1+\|X^0\|_{ L^p(\Omega;C^\gamma([0,T];H_0))}+\|X^0\|_{L^p(\Omega; L^\infty([0,T];H_0))}); \\
&\|U_n\|_{ L^p(\Omega;C^\gamma([0,T];H_0))}\le C(1+\|X^0_n\|_{ L^p(\Omega;C^\gamma([0,T];H_0))}+\|X^0_n\|_{L^p(\Omega; L^\infty([0,T];H_0))}).
\end{aligned}
\end{equation}
Indeed, we have that $U=X^0+\mathcal{L}_2(U)+\mathcal{L}_1(U)$ hence the claim for $U$ follows from the second inequality in \eqref{eq:smoothing} and \eqref{eq:stability}; with an analogous argument  showing the claim for $U_n$.
\end{remark}


Now we are ready to present our main result.

\begin{theorem}\label{thm:main}
Let $p>2$ and let  Assumption \ref{hypogeneral}, Assumption \ref{ass:semigroup}, and Assumption \ref{ass:semigroup:app}
 be satisfied with $0<\gamma_1<\frac12 -\frac{1}{p}$ and $0<\gamma_2<1$. Suppose that there exists a number $K>0$, such that
 $$
 \|X_n^0\|_{L^p(\Omega; L^\infty([0,T];H_0))}<K \quad \mbox{for all}\quad n\in \N.
 $$
Let
$$\Err(t) := U(t)-{U}^{(\h  )}(t)\quad \mbox{and}\quad \Err_0(t):= X_0(t)-X_0^n(t),\quad  t\in[0,T],
$$
 and let $0<\gamma<\min(\frac 12 -\frac 1p-\gamma_1,1 -\gamma_2)$.

Then there exists a constant $C>0$, depending on $\gamma,\gamma_1,\gamma_2$, $p$ and $T$,  such that for all $n\in\NN$ we have
\begin{equation}\label{eq:mee}
 \|e\|_{L^p(\Omega;C^\gamma([0,T];H_0))}\le C\lk( \|e_0\|_{L^p(\Omega;C^\gamma([0,T];H_0))} +\|e(0)\|_{L^p(\Omega;H_0)}+r_1(n)+r_2(n)\rk),
\end{equation}
where the rate functions $r_i$, $i=1,2$, are introduced in Assumption \ref{ass:semigroup:app}.
\end{theorem}

\begin{proof}
Let  $t \in [0,T]$ and fix a parameter $\lambda>0$. Then
\DEQS
\label{pert.split}
\lqq{ e_{\lambda}(t):=e^{-\lambda t}\Err(t) :=  e^{-\lambda t}(U(t) - U_n(t)) \phantom{\Big|}}
\\
& =&    e^{-\lambda t}X_0(t)-  e^{-\lambda t}X_0^n(t)\\
&&\qquad \quad {} +
\int_0^{t}  e^{-\lambda (t-s)}{S}^2  (t-s)  e^{-\lambda s}[F(s,U(s))-F(s,U_n(s))]\,ds
\\
&& \qquad\quad{} +  \int_0^{t}e^{-\lambda (t-s)}[S^2(t-s)-S^2_n  (t-s)]e^{-\lambda s}F(s, U_n(s))\,ds
\\
& &\qquad\quad {}+  \int_0^{t}e^{-\lambda (t-s)}{S}^1  (t-s)e^{-\lambda s}[G(s,U(s))-G(s,U_n(s))]\,dW_H(s)
\\
& &\qquad\quad {}+
\int_0^{t}e^{-\lambda (t-s)}[S^1(t-s)-S^1_\h  (t-s)]e^{-\lambda s}G(s, U_n(s))\,dW_H(s)
\\
 & :=& e^{\lambda}_0(t)+ I_1(t)+I_2(t)+I_3(t)+I_4(t). \phantom{\Big|}
\EEQS
Note that, by Remark \ref{rem:reg}, we have that
$$\|e\|_{L^p(\Omega;C^\gamma([0,T];H_0))}<\infty
$$
and hence also
$$\|e_\lambda\|_{L^p(\Omega;C^\gamma([0,T];H_0))}<\infty.
$$
In the following we estimate term by term.
\paragraph{\bf Estimate of $I_1$:}
To estimate $I_1$, we will use
 Lemma \ref{prop:detHolder} by setting $Y_1=H_2$, $Y_2=H_0$,
$$[0,T]\ni t\mapsto \Psi(t) := e^{-\lambda t}S^2(t),
$$
and $\Phi(t):= e^{-\lambda t}[F(t,U(t))-F(t,U_n(t))]$.  We have that $\Phi\in L^ {p}(\Omega;L^{\infty}([0,T];H_2))$ since $F$ is Lipschitz continuous and $U,U_n\in L^ {p}(\Omega;L^{\infty}([0,T];H_0))$ by Lemma \ref{lem:eu}.
Due to Assumption \ref{ass:semigroup}-(2), we know that the assumptions of Lemma \ref{prop:detHolder} are satisfied with $g(t)=(2+\lambda t )e^{-\lambda t}s_2(t)$ and $\theta=\gamma_2$. Indeed,
we have that
$$
\dot \Psi(t)x=-\lambda e^{-\lambda t}S^2(t)x+e^{-\lambda t}\dot S^2(t)x
$$
Therefore, as $\gamma_2<1$,
\begin{align}
&t^{\gamma_2} \n \dot \Psi(t) x\n_{H_0} + \gamma_2 t^{\gamma_2-1} \n \Psi(t)x\n_{H_0}\leq  t^{\gamma_2} \n \dot \Psi(t) x\n_{H_0} +  t^{\gamma_2-1} \n \Psi(t)x\n_{H_0}\\
&\quad \leq \lambda e^{-\lambda t} t^{\gamma_2}\|S^2(t)x\|_{H_0}+e^{-\lambda t} t^{\gamma_2}\|\dot S^2(t)x\|_{H_0}+  e^{-\lambda t} t^{\gamma_2-1}\|S^2(t)x\|_{H_0}\\
&\quad \leq \lambda t e^{-\lambda t} s_2(t)\|x\|_{H_2}+e^{-\lambda t} s_2(t)\|x\|_{H_2} +  e^{-\lambda t} s_2(t)\|x\|_{H_2}
=(2+\lambda t)e^{-\lambda t}s_2(t)\|x\|_{H_2}.
\end{align}
Thus, we can infer \fahim{by Lemma \ref{prop:detHolder}}
\begin{multline}
\|I_1\|_{L^p(\Omega;C^\gamma([0,T];H_0))} \\
 \le C \int_0^ T (2+\lambda t) e^{-\lambda t}s_2(t)\, dt \lk( \EE\esssup_{t\in [0,T]}e^{-\lambda p t}\|F(t,U(t))-F(t,U_n(t))\| ^{p}_{H_2} \rk) ^ \frac 1 {p}.
\end{multline}
The Lipschitz continuity of $F$ then gives
\begin{equation}\label{eq:I1}
\begin{aligned}
&\|I_1\|_{L^p(\Omega;C^\gamma([0,T];H_0))}\\
 &\quad \leq C\int_0^ T (2+\lambda t) e^{-\lambda t}s_2(t)\, dt\, \lk(\EE \esssup_{t\in [0,T]} \lk(e^{-\lambda p t}\|e(t)\|^p _{H_0}\rk)\rk)^{\frac 1 p}\\
 &=\quad C\int_0^ T (2+\lambda t) e^{-\lambda t}s_2(t)\, dt \,\lk(\EE \lk(\esssup_{t\in [0,T]} \|e_{\lambda}(t)\| _{H_0}\rk)^p\rk)^{\frac 1 p}\\
 &\quad \leq C\int_0^ T (2+\lambda t) e^{-\lambda t}s_2(t)\, dt\, \lk(\EE \lk(\esssup_{t\in [0,T]}\lk(t^{\gamma}\left\|\frac{e_\lambda(t)-e_{\lambda}(0)}{t^\gamma}\right\| _{H_0}+\|e(0)\|_{H_0}\rk)\rk)^{p}\rk)^{\frac 1p}\\
 & \quad \leq C \int_0^ T (2+\lambda t) e^{-\lambda t}s_2(t)\, dt \lk(T^\gamma \|e_\lambda\|_{L^p(\Omega;C^\gamma([0,T];H_0))}+\|e(0)\|_{L^p(\Omega;H_0)}\rk).
\end{aligned}
\end{equation}
\paragraph{\bf Estimate of $I_2$:}
To estimate $I_2$ we will use again Lemma \ref{prop:detHolder} by setting $Y_1=H_2$, $Y_2=H_0$,
$$[0,T]\ni t\mapsto \Psi(t) :=\frac{1}{r_2(n)}
 e^{-\lambda t}\lk[ S^2(t)-S^2_\h(t)\rk]= e^{-\lambda t}\Psi_n^2(t),
$$
and $\Phi(t):= e^{-\lambda t}F(t,U_n(t))$.
We have that $\Phi\in L^ {p}(\Omega;L^{\infty}([0,T];H_2))$ since $F$ is of linear growth in the second variable, uniformly in $t\in [0,T]$, and since we also have that $U_n\in L^ {p}(\Omega;L^{\infty}([0,T];H_0))$ by Lemma \ref{lem:eu}. Assumption \ref{ass:semigroup:app}-(2)  implies that the assumptions of Lemma \ref{prop:detHolder} are satisfied for all $n\in \N$ with  $g(t)=(2+\lambda t )e^{-\lambda t}h_2(t)$ and $\theta=\gamma_2$ as a similar calculation as in the case of $I_1$ shows. Therefore, we can infer
 from Lemma  \ref{prop:detHolder}
 that for all $\h\in\NN$,
\begin{align*}
&\frac{1}{r_2(n)}\|I_2\|_{L^p(\Omega;C^\gamma([0,T];H_0))} \\
&\qquad\le C \int_0^ T (2+\lambda t )e^{-\lambda t}h_2(t)\, dt \lk( \EE\esssup_{t\in [0,T]}e^{-\lambda p t} \|F(t, U_n(t))\|_{H_2}^ {p}\rk) ^ \frac 1 {p}\\
&\qquad\le C \int_0^ T h_2(t)\, dt \lk( \EE\esssup_{t\in [0,T]} \|F(t, U_n(t))\|_{H_2}^ {p} \rk) ^ \frac 1 {p}.
\end{align*}
The linear growth of $F$ gives for all $\h\in\NN$
$$
\frac{1}{r_2(n)}\|I_2\|_{L^p(\Omega;C^\gamma([0,T];H_0))}
\le C \int_0^ T h_2(t)\, dt \lk( 1+\EE \esssup_{t\in [0,T]}\|U_n(t)\|_{H_0}^ {p} \rk) ^ \frac 1 {p}.
$$
\paragraph{\bf Estimate of $I_3$:}
Here, we will apply  Lemma \ref{lem:stochConv} with
$Y_1=H_1$, $Y_2=H_0$,
$$[0,T]\ni t\mapsto \Psi(t) := e^{-\lambda t}S^1(t),
$$
and $\Phi(t):= e^{-\lambda t}[G(t,U(t))-G(t,U_n(t))]$, $t\in[0,T]$. We have that the process $\Phi$ is predictable as  $G$ is Lipschitz continuous and $U,U_n$ are predictable by Lemma \ref{lem:eu} and also that $\Phi\in L^{p}(\Omega;L^{\infty}([0,T];L_{HS}(H,H_1)))$  as $U,U_n\in L^ {p}(\Omega;L^{\infty}([0,T];H_0))$ again by Lemma \ref{lem:eu}.
Due to Assumption \ref{ass:semigroup}-(1) the assumptions of Lemma  \ref{lem:stochConv}, with $g(t)=(2+\lambda t )e^{-\lambda t}s_1(t)$ and $\theta=\gamma_1$ are satisfied, as a calculation similar to that in the case of $I_1$ shows.
Hence, it follows from  Lemma  \ref{lem:stochConv}, that
\begin{multline*}
\|I_3\|_{L^p(\Omega;C^\gamma([0,T];H_0))} \\
\le C \int_0^ T (2+\lambda t )e^{-\lambda t}s_1(t)\, dt \lk( \EE\esssup_{t\in [0,T]} e^{-\lambda p t}\|G(t,U(t))-G(t,U_n(t))\|_{L_{HS} (H, H_1)}^ {p} \rk) ^ \frac 1 {p}.
\end{multline*}
The Lipschitz continuity of $G$ then gives
\DEQS
\|I_3\|_{L^p(\Omega;C^\gamma([0,T];H_0))}  \le C \int_0^ T (2+\lambda t )e^{-\lambda t}s_1(t)\, dt\, \lk( \EE\esssup_{t\in [0,T]} e^{-\lambda p t}\|e(t)\|^ {p}_{H_0} \rk) ^ \frac 1 {p}.
\EEQS
By the same calculation used for estimating $I_1$ in \eqref{eq:I1}, where $s_2$ is replaced by $s_1$, we get
\begin{multline*}
\|I_3\|_{L^p(\Omega;C^\gamma([0,T];H_0))}  \\
\le C \int_0^ T (2+\lambda t) e^{-\lambda t}s_1(t)\, dt \lk(T^\gamma \|e_\lambda\|_{L^p(\Omega;C^\gamma([0,T];H_0))}+\|e(0)\|_{L^p(\Omega;H_0)}\rk).
\end{multline*}

\paragraph{\bf Estimate of $I_4$:}
To estimate $I_4$ again we will use Lemma \ref{lem:stochConv} by setting $Y_1=H_1$, $Y_2=H_0$,
$$[0,T]\ni t\mapsto \Psi(t) :=\frac{1}{r_1(n)}e^{-\lambda t}
 \lk[ S^1(t)-S^1_\h(t)\rk]=e^{-\lambda t}\Psi_n^1(t),
$$
and $\Phi(t):= e^{-\lambda t}G(t,U_n(t))$. The process $\Phi$ is predictable as $G$ is Lipschitz continuous and $U_n$ is predictable  by Lemma \ref{lem:eu} and $\Phi\in L^{p}(\Omega;L^{\infty}([0,T];L_{HS}(H,H_1)))$ as  $G$ is of linear growth in the second variable, uniformly in $[0,T]$,  and $U_n\in L^ {p}(\Omega;L^{\infty}([0,T];H_0))$ by Lemma \ref{lem:eu}.
Assumption \ref{ass:semigroup:app}-(1) implies that the assumption of Lemma \ref{lem:stochConv} are satisfied for all $n\in \N$ with $g(t)=(2+\lambda t )e^{-\lambda t}h_1(t)$ and $\theta=\gamma_1$, as a calculation similar to that in the case of $I_1$ shows.
Hence, we can infer that  for all $\h\in\NN$,
\begin{align*}
&\frac{1}{r_1(n)}\|I_4\|_{L^p(\Omega;C^\gamma([0,T];H_0))} \\
&\qquad\le C \int_0^ T(2+\lambda t )e^{-\lambda t}h_1(t)\, dt
\lk( \EE\esssup_{t\in [0,T]} e^{-\lambda p t}\|G(t,U_n(t))\|_{L_{HS} (H, H_1)}^ { p} \rk) ^ \frac1{p}\\
&\qquad\le C \int_0^ Th_1(t)\, dt
\lk( \EE\esssup_{t\in [0,T]} \|G(t,U_n(t))\|_{L_{HS} (H, H_1)}^ { p}\rk) ^ \frac1{p}.
\end{align*}
The linear growth of $G$ then gives,  for all $\h\in\NN$,
$$
\frac{1}{r_1(n)}\|I_4\|_{L^p(\Omega;C^\gamma([0,T];H_0))}
\le C \int_0^ T h_1(t)\, dt \lk(1+ \EE\esssup_{t\in [0,T]} \|U_n(t)\|_{H_0}^ { p}  \rk) ^ \frac 1 {p}.
$$

\medskip
\noindent
In this way
we get
\fahimm{
\begin{equation}\label{eq:el1}
\begin{aligned}
 &\|e_\lambda\|_{L^p(\Omega;C^\gamma([0,T];H_0))}\le  \|e^\lambda_0\|_{L^p(\Omega;C^\gamma([0,T];H_0)) }\\
 &\quad+C\lk(r_1(n)+r_2(n)\rk)\lk(1+\|U_n\|_{L^p(\Omega; L^\infty([0,T];H_0))}\rk)\\
 &\quad+C \int_0^ T (2+\lambda t) e^{-\lambda t}(s_1(t)+s_2(t))\, dt\,\Big(\|e_\lambda\|_{L^p(\Omega;C^\gamma([0,T];H_0))}+\|e(0)\|_{L^p(\Omega;H_0)}\Big) .
\end{aligned}
\end{equation}}
Note that, by Lemma \ref{lem:eu} and by our assumption $\|X_n^0\|_{L^p(\Omega; L^\infty([0,T];H_0))}<K$ for all $n\in \N$, there exists a constant $C>0$ such that
$$
\|U_n\|_{L^p(\Omega; L^\infty([0,T];H_0))}\le C\text{ for all } n\in \N.
$$
Furthermore, using Lebesgue's Dominated Convergence Theorem, it follows that
$$
\int_0^ T (2+\lambda t) e^{-\lambda t}(s_1(t)+s_2(t))\, dt \to 0 \text{ as }\lambda \to \infty,
$$
and, hence, with $\lambda>0$ large enough, we may absorb the last term on the right hand side of \eqref{eq:el1} into the left hand side to conclude that
\begin{equation}\label{eq:el2}
 \|e_\lambda\|_{L^p(\Omega;C^\gamma([0,T];H_0))}\le C\lk( \|e^{\lambda}_0\|_{L^p(\Omega;C^\gamma([0,T];H_0)) }+ \|e(0)\|_{L^p(\Omega;H_0)}+r_1(n)+r_2(n)\rk).
 \end{equation}
Remembering that $e_\lambda(0)= e(0)$ a straightforward  calculation shows that
\begin{equation*}
 \|e\|_{L^p(\Omega;C^\gamma([0,T];H_0))}
 \le C(\lambda,T)\lk( \|e_\lambda\|_{L^p(\Omega;C^\gamma([0,T];H_0))}+\|e(0)\|_{L^p(\Omega:H_0)}\rk).
\end{equation*}
A similar calculation implies that, remembering that $e_0(0)=e(0)$,
\begin{equation*}
 \|e^{\lambda}_0\|_{L^p(\Omega;C^\gamma([0,T];H_0))}
 \le C(\lambda,T)\lk( \|e_0\|_{L^p(\Omega;C^\gamma([0,T];H_0))}+\|e(0)\|_{L^p(\Omega:H_0)}\rk),
\end{equation*}
and the proof is complete in view of \eqref{eq:el2}.
\end{proof}
We end this section with the special important case $\gamma=0$.
\begin{corollary}\label{cor:main}
Let $p>2$ and let Assumptions \ref{hypogeneral} -- \ref{ass:semigroup:app} be satisfied with $0<\gamma_1<\frac12 -\frac{1}{p}$ and $0<\gamma_2<1$. Suppose that, there exists $K>0$, such that $\|X_n^0\|_{L^p(\Omega; L^\infty([0,T];H_0))}<K$ for all $n\in \N$.
Let $\Err(t) := U(t)-{U}^{(\h  )}(t)$ and $\Err_0(t):= X_0(t)-X_0^n(t)$, $t\in[0,T]$.
Then there exists a constant $C>0$, depending on $\gamma_1,\gamma_2$, $p$ and $T$,  such that, for all $n\in\NN$,
$$
 \|e\|_{L^p(\Omega;C([0,T];H_0))}\le C\lk( \|e_0\|_{L^p(\Omega;C([0,T];H_0))} +r_1(n)+r_2(n)\rk),
$$
where the rate functions $r_i$, $i=1,2$, are introduced in Assumption \ref{ass:semigroup:app}.
\end{corollary}
\begin{proof}
In a completely analogous fashion and using the same notation as in the proof of Theorem \ref{thm:main} using this time specifically item (c) from Lemma \ref{prop:detHolder} and \ref{lem:stochConv} one concludes that
\begin{equation}\label{eq:ell1}
\begin{aligned}
 &\|e_\lambda\|_{L^p(\Omega;C([0,T];H_0))}\le  \|e^\lambda_0\|_{L^p(\Omega;C([0,T];H_0)) }\\
 &\quad+C\lk(r_1(n)+r_2(n)\rk)\lk(1+\|U_n\|_{L^p(\Omega; L^\infty([0,T];H_0))}\rk)\\
 &\quad+C \int_0^ T (2+\lambda t) e^{-\lambda t}(s_1(t)+s_2(t))\, dt\,\|e_\lambda\|_{L^p(\Omega;C([0,T];H_0))} .
\end{aligned}
\end{equation}
Now the proof can be completed the same way as that of Theorem \ref{thm:main}.
\end{proof}

\section{Applications}\label{app}

In this section we give two typical instances where our abstract results are applicable.

\subsection{Spectral Galerkin approximations of a class of abstract stochastic integral equations}\label{sg_one}
Let $A:D(A)\subset H_0\to H_0$ be an unbounded, densely defined, self-adjoint, positive definite operator with compact inverse. Let $\lambda_n$ denote the eigenvalues of $A$, arranged in a non-decreasing order, with corresponding orthonormal eigenbasis $(e_n)\subset H_0$.
For $\alpha\in (0,2)$, $\beta>1/2$ and $\kappa>0$ we consider the integral equation introduced in \cite{DL} given by
\begin{multline}\label{eq:UUUU}
U(t)=S^{\alpha,1}(t)u_0+S^{\alpha,2}(t)u_1\\+\int_0^t S^{\alpha,\kappa}(t-s)F(U(s))\,d s+\int_0^tS^{\alpha,\beta}(t-s)G(U(s))\,d W_H(s)
\end{multline}
where the Laplace transform $\widehat{S^{\alpha,\beta}}(z)x:=\int_0^\infty e^{-zt}S^{\alpha,\beta}(t)x\,dt$, $x\in H_0$, $\Re z >0$, of $S^{\alpha,\beta}$ is given by 
\begin{equation}\label{eq:shatab}
\widehat{S^{\alpha,\beta}}(z)x=z^{\alpha-\beta}(z^{\alpha}+A)^{-1}x
\end{equation}
and $u_0,u_1\in L^p(\Omega;H_0)$ for some $p> 2$ are $\mathcal{F}_0$-measurable. To connect $S^{\alpha,\beta}$ to the convolution kernels $t\mapsto \frac{1}{\Gamma(\alpha)}t^{\alpha-1}$ and  $t\mapsto \frac{1}{\Gamma(\beta)}t^{\beta-1}$  we note that
it is shown in \cite[Lemma 5.4]{DL} that for $\alpha\in (0,2)$, $\beta>0$, and $x\in H_0$ one has
\begin{equation}\label{eq:sabb}
S^{\alpha,\beta}(t)x=A\int_0^t \frac{(t-s)^{\alpha-1}}{\Gamma(\alpha)} S^{\alpha,\beta }(s)x\,ds+ \frac{1}{\Gamma(\beta)}t^{\beta-1}x.
\end{equation}
Next we provide smoothing estimates for $S^{\alpha,\beta}$ and its derivative.
\begin{lemma}\label{lem:sm}
For $\xi\in [0,1]$, $\alpha\in (0,2)$ and $\beta>0$ the estimates
\begin{align}
\|A^\xi S^{\alpha,\beta}(t)\|_{\mathcal{L}(H_0)}\leq Mt^{\beta-\alpha\xi-1},\,t>0,\label{eq:sabe}\\
\|A^\xi \dot{S}^{\alpha,\beta}(t)\|_{\mathcal{L}(H_0)}\leq Mt^{\beta-\alpha\xi-2},\,t>0,\label{eq:sabe1}
\end{align}
hold for some $M=M(\alpha,\beta,\xi)$.
\end{lemma}
\begin{proof}
It is shown  in \cite[Lemma 5.4]{DL} that for all $x\in H_0$,
$$
S^{\alpha,\beta}(t)x=\frac{1}{2\pi i}\int_{\Gamma_{\rho,\phi}}e^{zt}z^{\alpha-\beta}(z^\alpha+A)^{-1}x\,dz,
$$
where the contour is given by
$$
\Gamma_{\rho,\phi}(s)=
\begin{cases}
(s-\phi+\rho)e^{i\phi}&\text{ for }s>\phi,\\
\rho e^{is} &\text{ for } s\in (-\phi,\phi),\\
(-s-\phi+\rho)e^{-i\phi} &\text{ for }s<-\phi,
\end{cases}
$$
where $\rho>0$, $\phi>\frac{\pi}{2}$ and $\alpha \phi<\pi$. Therefore, as $A^{\xi}$ is a closed operator, one has that $S^{\alpha,\beta}(t)x\in \mathcal{D}(A^\xi)$ and
\begin{equation}\label{eq:axirep}
A^\xi S^{\alpha,\beta}(t)x=\frac{1}{2\pi i}\int_{\Gamma_{\rho,\phi}}e^{zt}z^{\alpha-\beta}A^{\xi}(z^\alpha+A)^{-1}x\,dz
\end{equation}
provided that
$$
\int_{\Gamma_{\rho,\phi}}\left|e^{zt}z^{\alpha-\beta}\right|\left\|A^\xi(z^\alpha+A)^{-1}x\right\|_{H_0}\,|dz|<\infty.
$$
It is shown in the proof of \cite[Lemma 5.4]{DL} that
\begin{equation}\label{eq:axies}
\int_{\Gamma_{\rho,\phi}}\left|e^{zt}z^{\alpha-\beta}\right|\left\|A^\xi(z^\alpha+A)^{-1}x\right\|_{H_0}\,|dz|\leq
Ct^{\beta-\alpha\xi-1}\|x\|_{H_0}\int_{\Gamma_{1, \phi}} e^{\Re z} |z|^{-\beta+\alpha \xi}\,|dz|,
\end{equation}
and hence \eqref{eq:axirep} holds and
$$
\|A^\xi S^{\alpha,\beta}(t)\|_{\mathcal{L}(H_0)}\leq C t^{\beta-\alpha\xi-1},\,t>0.
$$
To show \eqref{eq:sabe1} note that, by \cite[Lemma 5.4]{DL} the function $t\to S^{\alpha,\beta}(t)x$ can be extended analytically to a sector in the right half-plane for all $x\in H_0$. In particular, the function $t\to S^{\alpha,\beta}(t)x$ is differentiable for $t>0$.
Hence, we have
\begin{equation}\label{eq:axidotrep}
A^{\xi} \dot{S}^{\alpha,\beta}(t)x=\frac{1}{2\pi i}\int_{\Gamma_{\rho,\phi}}ze^{zt}z^{\alpha-\beta}A^{\xi}(z^\alpha+A)^{-1}x\,dz,
\end{equation}
provided that for every $t>0$ there is $\epsilon>0$ and $K=K(t,\epsilon)>0$ such that
$$
\int_{\Gamma_{\rho,\phi}}\left|ze^{zt}z^{\alpha-\beta}\right|\left\|A^{\xi}(z^\alpha+A)^{-1}x\right\|_{H_0}\,|dz| <K
$$
for $t\in (t_0-\epsilon,t_0+\epsilon)$. In a completely analogous fashion as in the case of estimate \eqref{eq:axies}, we get
\begin{multline*}
\int_{\Gamma_{\rho,\phi}}\left|ze^{zt}z^{\alpha-\beta}\right|\left\|A^{\xi}(z^\alpha+A)^{-1}x\right\|_{H_0}\,|dz| \\
\leq Ct^{\beta-\alpha\xi-2}\|x\|_{H_0}\int_{\Gamma_{1, \phi}} e^{\Re z} |z|^{-\beta+\alpha \xi+1}\,|dz|,
\end{multline*}
and thus \eqref{eq:axidotrep} holds and
$$
\|A^\xi \dot{S}^{\alpha,\beta}(t)\|_{\mathcal{L}(H_0)}\leq Mt^{\beta-\alpha\xi-2},\,t>0,
$$
which finishes the proof.
\end{proof}
\begin{remark}
We would like to point out two crucial points concerning \eqref{eq:sabe} and \eqref{eq:sabe1}. First, unless $\alpha=\beta=1$, which correspond the heat equation, the estimates do not hold for $\xi>1$. Furthermore, the constant $M$ in the estimates blows up as $\alpha\to 2$ as in this case we must have $\phi\to \pi/2$ and hence we integrate on a path with infinite line segments approaching the imaginary axis.
\end{remark}
\begin{ass}\label{hypogeneral1}
We assume that
\begin{enumerate}
\item[(a)] 
there exists $0\leq \delta_F\leq 1$ such that $F: H_0\to H_{-\delta_F}^A $ is Lipschitz continuous. In particular, there exists a constant $C>0$ with
$$
\| A^{-\delta_F}(F(x)-F(y))\|_{H_0}\le C\, \|x-y\|_{H_0},\quad x,y\in H_0.
$$
\item[(b)]
There exists $0\leq \delta_G\leq 1 $ such that
$G: H_0\to L_{HS}(H,H_{-\delta_G}^A )$ is Lipschitz continuous.
In particular, there exists a constant $C>0$ with
$$
\| A^{-\delta_G}(G(x)-G(y))\|_{L_{HS}(H_0,H_0)}\le C\, \|x-y\|_{H_0},\quad x,y\in H_0\fahim{.}
$$
\end{enumerate}
\end{ass}
Thus, the spaces $H_1$ and $H_2$ become
$$
H_1:=H_{-\delta_G}^A \quad \mbox{and}\quad H_2:=H_{-\delta_F}^A.
$$
Note next that $A^\xi$ commutes with $S^{\alpha,\beta}(t)$ for all $t\geq 0$ and all $\xi\in [-1,0]$ using an inversion formula for the Laplace transform (see, for example, \cite[Chapter 2.4]{ABHN01}) as this property clearly holds for $\widehat{S^{\alpha,\beta}}(z)$, $\Re z>0$, and $A^\xi\in \mathcal{L}(H_0)$ is this case.
Then, for $t>0$, \eqref{eq:sabe} and \eqref{eq:sabe1} show that
\begin{equation}\label{eq:ex1est1}
\begin{aligned}
&\|S^{\alpha,\beta}(t)x\|_{H_0}\leq Mt^{\beta-\alpha\delta_G-1}\|x\|_{H_1},\quad \|\dot{S}^{\alpha,\beta}(t)x\|_{H_0}\leq Mt^{\beta-\alpha\delta_G-2}\|x\|_{H_1};\\
&\|S^{\alpha,\kappa}(t)x\|_{H_0}\leq Mt^{\kappa-\alpha\delta_F-1}\|x\|_{H_2},\quad  \|\dot{S}^{\alpha,\kappa}(t)x\|_{H_0}\leq  Mt^{\kappa-\alpha\delta_F-2}\|x\|_{H_2}.
\end{aligned}
\end{equation}
We also define the approximating operators by
$$
S^{\alpha,\beta}_n(t):=S^{\alpha,\beta}(t)\mathcal{P}_n=\mathcal{P}_nS^{\alpha,\beta}(t),
$$
where $\mathcal{P}_n$ is defined in \eqref{PN}, and the approximating process by
\begin{equation}\label{eq:uuuun}
U_n(t)=S_n^{\alpha,1}(t)u_0+S_n^{\alpha,2}(t)u_1+\int_0^t S_n^{\alpha,\kappa}(t-s)F(U_n(s))\,d s+\int_0^tS_n^{\alpha,\beta}(t-s)G(U_n(s))\,d s.
\end{equation}
For $\nu\geq 0$, let
\begin{equation}\label{eq:psi111n}
\Psi^1_n(t):=\lambda_{n+1}^\nu(S^{\alpha,\beta}(t)-S_n^{\alpha,\beta}(t))=\lambda_{n+1}^\nu(I-\mathcal{P}_n)S^{\alpha,\beta}(t)
\end{equation}
and
\begin{equation}\label{eq:psi222n}
\Psi^2_n(t):=\lambda_{n+1}^\nu(S^{\alpha,\kappa}(t)-S_n^{\alpha,\kappa}(t))=\lambda_{n+1}^\nu(I-\mathcal{P}_n)S^{\alpha,\kappa}(t).
\end{equation}
For $t>0$, we then have for $\delta_G+\nu\leq 1$ and $\delta_F+\nu\leq 1$,
\begin{equation}\label{eq:ex1est2}
\begin{aligned}
&\|\Psi^1_n(t)x\|_{H_0}\leq C \|A^\nu S^{\alpha,\beta}(t)x\|_{H_0} \leq  Ct^{\beta-\alpha(\delta_G+\nu)-1}\|x\|_{H_1};\\
&\|\dot{\Psi}^1_n(t)x\|_{H_0}\leq C \|A^\nu \dot{S}^{\alpha,\beta}(t)x\|_{H_0} \leq  C t^{\beta-\alpha(\delta_G+\nu)-2}\|x\|_{H_1};\\
&\|\Psi^2_n(t)x\|_{H_0}\leq C\|A^\nu S^{\alpha,\kappa}(t)x\|_{H_0} \leq  C t^{\kappa-\alpha(\delta_F+\nu)-1}\|x\|_{H_2};\\
&\|\dot{\Psi}^2_n(t)x\|_{H_0}\leq C \|A^\nu \dot{S}^{\alpha,\kappa}(t)x\|_{H_0} \leq  C t^{\kappa-\alpha(\delta_F+\nu)-2}\|x\|_{H_2}.
\end{aligned}
\end{equation}
\begin{theorem}\label{thm:example1}
Let $U$ and $\{U_n:n\in\NN\}$ given by \eqref{eq:UUUU} and \eqref{eq:uuuun}, respectively. Let $\Err(t) := U(t)-U_n(t)$ and $e_0(t):=S^{\alpha,1}(t)u_0-S^{\alpha,1}_n(t)u_0+S^{\alpha,2}(t)u_1-S^{\alpha,2}_n(t)u_1.$
\begin{itemize}
\item[(a)]Let $p>2$, $0< \gamma_1<\frac12-\frac{1}{p}$ and $0< \gamma_2<1$ and suppose that $\min(\gamma_1+\beta-\alpha \delta_G-1,\gamma_2+\kappa-\alpha \delta_F-1)>0$.
Let $\gamma<\min(\frac12-\frac{1}{p}-\gamma_1,1-\gamma_2)$, $\nu_1<\frac{\gamma_1+\beta -\alpha\delta_G-1}{\alpha}$  and $\nu_2<\frac{\gamma_2+\kappa -\alpha\delta_F-1}{\alpha}$. If $\delta_G+\nu_1\leq 1$ and $\delta_F+\nu_2\leq 1$, then the error estimate
\begin{multline}\label{eq:sgest}
 \|e\|_{L^p(\Omega;C^\gamma([0,T];H_0))}\leq C(T,p, \nu_1,\nu_2,\gamma)\big( \|e_0\|_{L^p(\Omega;C^\gamma([0,T];H_0))} \\+\|e(0)\|_{L^p(\Omega;H_0)}+\lambda_{n+1}^{-\nu_1}+\lambda_{n+1}^{-\nu_2}\big)
\end{multline}
holds. Set $\nu:=\min(\nu_1,\nu_2)$ and suppose that $\nu+\frac{\gamma}{\alpha}\leq 1$. If $u_0\in L^p(\Omega;\D(A^{\nu+\frac{\gamma}{\alpha}}))$ and $u_1\in L^p(\Omega;\D(A^{\max(0,\nu+\frac{\gamma-1}{\alpha})}))$, then
$$
 \|e\|_{L^p(\Omega;C^\gamma([0,T];H_0))}\leq C(T,p, \nu_1,\nu_2,\gamma, u_0,u_1)\lambda_{n+1}^{-\nu}.
$$
\item[(b)] Let $p>2$ and suppose that $\min(\frac12-\frac{1}{p}+\beta-\alpha \delta_G-1,\kappa-\alpha \delta_F)>0$. For $\nu_1<\frac{\frac12-\frac{1}{p}+\beta -\alpha\delta_G-1}{\alpha}$  and $\nu_2<\frac{\kappa -\alpha\delta_F}{\alpha}$, if $\delta_G+\nu_1\leq 1$ and $\delta_F+\nu_2\leq 1$, then the error estimate
\begin{equation}\label{eq:sgest2}
 \|e\|_{L^p(\Omega;C([0,T];H_0))}\leq C(T,p, \nu_1,\nu_2)\big( \|e_0\|_{L^p(\Omega;C([0,T];H_0))} +\lambda_{n+1}^{-\nu_1}+\lambda_{n+1}^{-\nu_2}\big)
\end{equation}
holds. Setting $\nu:=\min(\nu_1,\nu_2)$, and assuming that $u_0\in L^p(\Omega;\D(A^{\nu}))$ and $u_1\in L^p(\Omega;\D(A^{\max(0,\nu-\frac{1}{\alpha})}))$ the error estimate
$$
 \|e\|_{L^p(\Omega;C([0,T];H_0))}\leq C(T,p, \nu_1,\nu_2, u_0,u_1)\lambda_{n+1}^{-\nu}
$$
holds.
\end{itemize}
\end{theorem}
\begin{proof}
Estimate \eqref{eq:sgest} follows from Theorem \ref{thm:main} using Assumption \ref{hypogeneral1} and estimates \eqref{eq:ex1est1} and \eqref{eq:ex1est2}.  To estimate the initial terms first note that as $S^{\alpha,1}(0)=I$ and $S^{\alpha,2}(0)=0$ it follows that
$$
\|e(0)\|_{L^p(\Omega;H_0)}=\|u_0-\mathcal{P}_nu_0\|_{L^p(\Omega;H_0)}\leq C \lambda_{n+1}^{-\nu}\|A^\nu u_0\|_{L^p(\Omega;H_0)}.
$$
It is shown in \cite[Lemma 5.4]{DL} that if $x\in \D(A^{\nu+\frac{\gamma}{\alpha}})$ with $\nu+\frac{\gamma}{\alpha}\leq 1$, then the function $v(t):= S^{\alpha,1}(t)x-x$ admits a fractional derivative of order $\gamma$ defined by
$$
D^{\gamma}_tv(t)=\frac{1}{\Gamma(1-\gamma)}\frac{d}{dt}\int_0^t(t-s)^{-\gamma}v(s)\,d s,\,\gamma\in (0,1),
$$
in $\mathcal{D}(A^\nu)$ and
$$
\|A^\nu D^{\gamma}_tv(t)\|_{H_0}\leq M\|A^{\nu+\frac{\gamma}{\alpha}}x\|_{H_0}, \,t\in (0,T].
$$
A straightforward calculation shows that (see also, \cite[Vol. II, p. 138]{Zygmund}, \cite{CLS})
$$
\|A^\nu v(\cdot)\|_{C^\gamma([0,T];H_0)} \leq C \|A^\nu D^{\gamma}_tv(\cdot)\|_{L^{\infty}([0,T];H_0)}.
$$
Therefore,
\begin{align}
&\|S^{\alpha,1}(\cdot)u_0-S^{\alpha,1}_n(\cdot)u_0\|_{L^p(\Omega;C^\gamma([0,T];H_0))}=\|S^{\alpha,1}(\cdot)(I-\mathcal{P}_n)u_0\|_{L^p(\Omega;C^\gamma([0,T];H_0))}\\
&\leq \lambda_{n+1}^{-\nu}\|A^\nu S^{\alpha,1}(\cdot)u_0\|_{L^p(\Omega;C^\gamma([0,T];H_0))}=\lambda_{n+1}^{-\nu}\|A^\nu (S^{\alpha,1}(\cdot)u_0-u_0)\|_{L^p(\Omega;C^\gamma([0,T];H_0))}\\
&\leq C  \lambda_{n+1}^{-\nu}\|A^\nu D^{\gamma}_t (S^{\alpha,1}(\cdot)u_0-u_0)\|_{L^p(\Omega;L^{\infty}([0,T];H_0))}
\leq C  \lambda_{n+1}^{-\nu} \|A^{\nu+\frac{\gamma}{\alpha}} u_0\|_{L^p(\Omega;H_0)}.
\end{align}
Here, for the equality in the second row of the calculation, we used the fact that the $C^\gamma([0,T];H_0))$-seminorm of a function $f:[0,T]\to H_0$ does not change by adding a constant to $f$.
Similarly, it is also shown in
 \cite[Lemma 5.4]{DL} that if $x\in \D(A^{\max(0,\nu+\frac{\gamma-1}{\alpha})})$, then the function $v(t):= S^{\alpha,2}(t)x$ admits a fractional derivative in $\mathcal{D}(A^\nu)$ and
$$
\|A^\nu D^{\gamma}_tv(t)\|_{H_0}\leq M\|A^{\max(0,\nu+\frac{\gamma-1}{\alpha})}x\|_{H_0}, \,t\in (0,T].
$$
Then, similarly as above we conclude that
$$
\|S^{\alpha,2}(\cdot)u_1-S^{\alpha,2}_n(\cdot)u_1\|_{L^p(\Omega;C^\gamma([0,T];H_0))}\leq C \lambda_{n+1}^{-\nu} \|A^{\max(0,\nu+\frac{\gamma-1}{\alpha})} u_1\|_{L^p(\Omega;H_0)}
$$
and the proof of (a) is complete. To show the claims in (b) first note that estimate \eqref{eq:sgest2} follows from Corollary \ref{cor:main} by choosing $\gamma_1$ sufficiently close to $\frac{1}{2}-\frac{1}{p}$ and $\gamma_2$ sufficiently close to 1 together with Assumption \ref{hypogeneral1} and estimates \eqref{eq:ex1est1} and \eqref{eq:ex1est2}. To estimate the initial terms first note that, by \eqref{eq:sabe} with $\xi=0$ and $\beta=1$, using also the fact that $S^{\alpha,1}$ and $A^{-\nu}$ commutes for $\nu\in [0,1]$,
\begin{multline}
\|S^{\alpha,1}(t)u_0-S^{\alpha,1}_n(t)u_0\|_{H_0}=\|(I-\mathcal{P}_n)S^{\alpha,1}(t)u_0\|_{H_0}\leq C \lambda_{n+1}^{-\nu}\|A^\nu S^{\alpha,1}(t)u_0\|_{H_0}\\
=C \lambda_{n+1}^{-\nu}\|A^\nu S^{\alpha,1}(t)A^{-\nu}A^\nu u_0\|_{H_0}=C \lambda_{n+1}^{-\nu}\| S^{\alpha,1}(t)A^\nu u_0\|_{H_0}\leq C \lambda_{n+1}^{-\nu} \|A^\nu u_0\|_{H_0},
\end{multline}
for all $t\geq 0$. Therefore,
$$
\|S^{\alpha,1}(\cdot)u_0-S^{\alpha,1}_n(\cdot)u_0\|_{L^p(\Omega;C([0,T];H_0))}\leq C  \lambda_{n+1}^{-\nu} \|A^{\nu} u_0\|_{L^p(\Omega;H_0)}.
$$
For the second initial term, let first $\nu\leq \frac{1}{\alpha}$. Then, by \eqref{eq:sabe} with $\xi=\nu$ and $\beta=2$,
 \begin{multline}
\|S^{\alpha,2}(t)u_1-S^{\alpha,2}_n(t)u_1\|_{H_0}=\|(I-\mathcal{P}_n)S^{\alpha,2}(t)u_1\|_{H_0}\leq C \lambda_{n+1}^{-\nu}\|A^\nu S^{\alpha,2}(t)u_1\|_{H_0}\\
=C t^{1-\alpha \nu}\lambda_{n+1}^{-\nu}\|u_1\|_{H_0},\, t\geq 0.
\end{multline}
On the other hand, if $1\geq \nu> \frac{1}{\alpha}$, then by \eqref{eq:sabe} with $\xi=\frac{1}{\alpha}$ and $\beta=2$,
\begin{multline*}
\|S^{\alpha,2}(t)u_1-S^{\alpha,2}_n(t)u_1\|_{H_0}=\|(I-\mathcal{P}_n)S^{\alpha,2}(t)u_1\|_{H_0}\leq C \lambda_{n+1}^{-\nu}\|A^\nu S^{\alpha,2}(t)u_1\|_{H_0}\\
=C \lambda_{n+1}^{-\nu}\|A^{\frac{1}{\alpha}} S^{\alpha,2}(t)A^{\nu-\frac{1}{\alpha}}u_1\|_{H_0}\leq C \lambda_{n+1}^{-\nu}\|A^{\nu-\frac{1}{\alpha}}u_1\|_{H_0},\,t\geq 0.
\end{multline*}
In summary, if $u_1\in L^p(\Omega;\D(A^{\max(0,\nu-\frac{1}{\alpha})}))$, then
$$
\|S^{\alpha,2}(\cdot)u_1-S^{\alpha,2}_n(\cdot)u_1\|_{L^p(\Omega;C^\gamma([0,T];H_0))}\leq C \lambda_{n+1}^{-\nu} \|A^{\max(0,\nu-\frac{1}{\alpha})} u_1\|_{L^p(\Omega;H_0)},
$$
which finishes the proof of item (b) and hence that of the theorem.
\end{proof}
\begin{remark}
In general, we may observe that larger values of \fahimm{the parameters $\beta$ and $\kappa$} allow for higher rates of convergence in Theorem \ref{thm:example1} as larger values of these parameters correspond to better time-regularity of the stochastic, respectively, deterministic feedback. Note also, that the operator family $S^{\alpha,2}$, which is $S^{\alpha,\beta}$ with $\beta=2$ has a stronger smoothing effect than $S^{\alpha,1}$, which is $S^{\alpha,\beta}$ with $\beta=1$, as Lemma \ref{lem:sm} shows. This explains why the regularity requirement in Theorem \ref{thm:example1} is stricter on $u_0$ than that on $u_1$ for the same rate of convergence.
\end{remark}
\begin{example}[Stochastic heat equation]\label{ex:sg1}
The stochastic heat equation corresponds to parameters $\alpha=\beta=\kappa=1$ and $u_1=0$. In this case, Theorem \ref{thm:example1}, yields
$$
 \|e\|_{L^p(\Omega;C^\gamma([0,T];H_0))}\leq C\lambda_{n+1}^{-\nu},
$$
for $\gamma<\min(\frac12-\frac{1}{p}-\gamma_1,1-\gamma_2)$ and $\nu=\min(\nu_1,\nu_2)$ where $\nu_1<\gamma_1 -\delta_G$  and $\nu_2<\gamma_2 -\delta_F$, provided that $u_0\in L^p(\Omega;\D(A^{\nu+\gamma}))$. Furthermore,
$$
 \|e\|_{L^p(\Omega;C([0,T];H_0))}\leq  C\lambda_{n+1}^{-\nu}
$$
for $\nu=\min(\nu_1,\nu_2)$ where $\nu_1< \frac12-\frac{1}{p}-\delta_G$  and $\nu_2<1 -\delta_F$, provided that $u_0\in L^p(\Omega;\D(A^{\nu}))$. This is consistent with  \cite[Proposition 3.1]{sonja3}.
\end{example}
\begin{example}[Fractional stochastic heat equation]\label{ex:sg2} The simplest fractional stochastic heat equation, considered also in Example \ref{ex:main1} in the standard global Lipschitz setting, corresponds to parameters $\beta=\kappa=1$, $\alpha\in (0,1)$ and $u_1=0$. In this case, Theorem \ref{thm:example1}, yields
$$
 \|e\|_{L^p(\Omega;C^\gamma([0,T];H_0))}\leq C\lambda_{n+1}^{-\nu},
$$
for $\gamma<\min(\frac12-\frac{1}{p}-\gamma_1,1-\gamma_2)$ and $\nu=\min(\nu_1,\nu_2)$ where $\nu_1<\frac{\gamma_1}{\alpha} -\delta_G$  and $\nu_2<\frac{\gamma_2}{\alpha} -\delta_F$, provided that $u_0\in L^p(\Omega;\D(A^{\nu+\frac{\gamma}{\alpha}}))$ and $\max(\nu_1+\delta_G,\nu_2+\delta_F)\leq 1$ holds. Furthermore,
$$
 \|e\|_{L^p(\Omega;C([0,T];H_0))}\leq  C\lambda_{n+1}^{-\nu}
$$
for $\nu=\min(\nu_1,\nu_2)$ where $\nu_1< \frac{\frac12-\frac{1}{p}}{\alpha}-\delta_G$  and $\nu_2<\frac{1}{\alpha} -\delta_F$, provided that $u_0\in L^p(\Omega;\D(A^{\nu}))$ and $\max(\nu_1+\delta_G,\nu_2+\delta_F)\leq 1$ holds. As discussed in  Example \ref{ex:main1}, the rate only improves with decreasing $\alpha$ as long as $\max(\nu_1+\delta_G,\nu_2+\delta_F)\leq 1$ holds and stays the same when $\alpha$ decreases further. For example, if $\delta_F=\delta_G=0$, then
$$
 \|e\|_{L^p(\Omega;C([0,T];H_0))}\leq  C\lambda_{n+1}^{-\nu}
$$
for $\nu<\max(\frac{\frac12-\frac{1}{p}}{\alpha},1)$.
\end{example}
\begin{example}[Fractional stochastic wave equation]\label{ex:sg3}
The simplest fractional stochastic wave equation, considered also in Example \ref{ex:main2} in the standard global Lipschitz setting, corresponds to parameters $\beta=\kappa=1$, $\alpha\in (1,2)$. In this case, Theorem \ref{thm:example1}, yields
$$
 \|e\|_{L^p(\Omega;C^\gamma([0,T];H_0))}\leq C\lambda_{n+1}^{-\nu},
$$
for $\gamma<\min(\frac12-\frac{1}{p}-\gamma_1,1-\gamma_2)$ and $\nu=\min(\nu_1,\nu_2)$ where $\nu_1<\frac{\gamma_1}{\alpha} -\delta_G$  and $\nu_2<\frac{\gamma_2}{\alpha} -\delta_F$, provided that $u_0\in L^p(\Omega;\D(A^{\nu+\frac{\gamma}{\alpha}}))$ and $u_1\in L^p(\Omega;\D(A^{\max(0,\nu+\frac{\gamma-1}{\alpha})}))$. Furthermore,
$$
 \|e\|_{L^p(\Omega;C([0,T];H_0))}\leq  C\lambda_{n+1}^{-\nu}
$$
for $\nu=\min(\nu_1,\nu_2)$ where $\nu_1< \frac{\frac12-\frac{1}{p}}{\alpha}-\delta_G$  and $\nu_2<\frac{1}{\alpha} -\delta_F$, provided that $u_0\in L^p(\Omega;\D(A^{\nu}))$, $u_1\in L^p(\Omega;\D(A^{\max(0,\nu-\frac{1}{\alpha})}))$. We see that the rates deteriorate with increasing $\alpha$.
\end{example}

\subsection{Finite element approximation of a stochastic fractional wave equation}\label{sg_two}
In this section we give a more concrete example to demonstrate how the abstract framework can be used for finite element approximation.  Let $\mathcal{D}\subset \R^d$ be a bounded convex polygonal domain and let $A=-\Delta$ be the Dirichlet Laplacian in $(H_0,\|\cdot\|_{H_0}):=(L^2(\mathcal{D}),\|\cdot\|_{L^2(\mathcal{D})})$ with domain $\mathcal{D}(A)=H^2(\mathcal{D})\cap H^1_0(\mathcal{D})$. Let $p>2$ and consider a fractional stochastic wave equation  \cite{CDP,KLS20} given by
\begin{equation}\label{eq:fwe}
\left\{
\begin{aligned}
dU(t)+A\int_0^tb(t-s)U(s)\,ds\,dt&=F(U(s))dt+R(U(s))Q^\frac{1}{2}d W_H(t),\,t>0; \\
U(0)&=U_0,
\end{aligned}
\right.
\end{equation}
where $W_H$ is a cylindrical Wiener process in $H=H_0$ and $Q:H\to H$ is a linear, symmetric, positive semidefinite, trace class operator on $H$ with an orthonormal basis of eigenfunctions $\{e_k:k\in\NN\}$ with $\|e_k\|_{L^\infty(\mathcal{D})}\leq M$ for all $k=1,2,\dots$. The initial data $U_0\in L^p(\Omega;H_0)$ for some $p>2$ is assumed to be $\mathcal{F}_0$-measurable. The kernel $b$ is the Riesz kernel given by
$$
b(t)=\frac{t^{\alpha-1}}{\Gamma(\alpha)},\quad \alpha \in (0,1).
$$
For $u\in H_0$ define $[F(u)](r):=f(u(r))$ with $f:\R\to\R$ being a globally Lipschitz continuous function while $[R(u)v](r):=g(u(r))v(r)$ with $g:\R\to\R$ being a globally Lipschitz continuous function. Then we may take $H_0=H_1=H_2=H$ and a straightforward calculation yields that $F$ and $G(u)v:=R(u)Q^\frac12 v$ satisfy Assumption \ref{hypogeneral}.
The mild solution of \eqref{eq:fwe} is given by the variation of constants formula
\begin{equation}\label{eq:UUU}
U(t)=S(t)U_0+\int_0^tS(t-s)F(U(s))\, ds+\int_0^tS(t-s)G(U(s))\,d W_H(s).
\end{equation}
Here the resolvent family $\{S(t)\}_{t\geq 0}$ is a strongly continuous family of bounded linear operators on $H_0$, which is strongly differentiable  on $(0,\infty)$ such that the function $t\mapsto S(t)x$ is the unique solution of
\begin{equation}\label{eq:vstro}
\dot{u}(t)+A\int_0^t b(t-s)u(s)\,\, d s=0,\,t>0;\,u(0)=x,
\end{equation}
see \cite[Corollary 1.2]{pruss}. 
\begin{remark}\label{rem:corresp}
In connection with Subsection \ref{sg_one}, in particular \eqref{eq:sabb}, by integrating \eqref{eq:vstro} from 0 to $t$,  one sees that in fact $S(t)=S^{\alpha+1,1}(t)$.
\end{remark}
 The resolvent family $S$ has the following smoothing properties \cite{MT}:
\begin{align}
&\|A^\mu S(t)\|_{\mathcal{L}(H_0)}\leq Ct^{-(\alpha+1)\mu} ,\quad \mu\in [0,1],\,t>0;\\
&\|A^\mu \dot{S}(t)\|_{\mathcal{L}(H_0)}\leq Ct^{-(\alpha+1)\mu-1} ,\quad \mu\in [-1,1],\,t>0.\label{eq:smp2}
\end{align}
For spatial approximation of \eqref{eq:fwe} we consider a standard continuous finite element method.  Let $\{\mathcal{T}_h\}_{0<h<1}$ denote a family of triangulations of $\mathcal{D}$, with mesh size $h>0$ and consider finite element spaces $\{ V_h \}_{0<h<1}$, where $V_h\subset H^1_0(\mathcal{D})$ consists of continuous piecewise linear functions vanishing at the boundary of $\mathcal{D}$. We introduce the "discrete Laplacian" (see, for example, \cite[page 10]{Thomeebook})
\begin{equation}\label{def:Ah}
  A_{h}:V_h\to V_h,   \quad
    \inner[H_0]{A_{h} \psi}{ \chi} = \inner[H_0]{ \nabla \psi}{\nabla \chi},\quad \psi,\chi \in V_h,
\end{equation}
where $\inner[H_0]{\cdot}{\cdot}$ denotes the inner product of $H_0$, and the orthogonal projection
$$
    P_{h}: H_0 \to V_h,\quad
   \inner[H_0]{P_{h} f}{ \chi} = \inner[H_0]{ f} {\chi},\quad \chi \in V_h.
$$
\fahim{We consider the approximated problem}
\begin{equation}\label{eq:fweh}
\left\{
\begin{aligned}
dU_h(t)+A_h\int_0^tb(t-s)U_h(s)\,ds\,dt&=P_hF(U_h(s))dt+P_hG(U_h(s))d W_H(t); \\
U_h(0)&=P_hU_0,
\end{aligned}
\right.
\end{equation}
with mild solution given by
\begin{equation}\label{eq:uh}
U_h(t)=S_h(t)P_hU_0+\int_0^tS_h(t-s)P_hF(U_h(s))\, ds+\int_0^tS_h(t-s)P_hG(U_h(s))\,d W_H(s).
\end{equation}
Similarly to the resolvent family $\{S(t)\}_{t\geq 0}$, the resolvent family $\{S_h(t)\}_{t\geq 0}$ is a strongly continuous family of bounded linear operators on $V_h$, which is strongly differentiable  on $(0,\infty)$ such that for $\chi\in V_h$ the $V_h$-valued function $t\mapsto S_h(t)\chi$ is the unique solution of
\begin{equation}\label{eq:vstroh}
\dot{u}_h(t)+\int_0^t b(t-s)A_hu_h(s)\,\, d s=0,\,t>0;\,u_h(0)=\chi.
\end{equation}
Let $E_h(t):=S(t)-S_h(t)P_h$ denote the deterministic error operator. We have the following error bounds.
\begin{proposition}\label{prop:estimateEh}
Let $\epsilon>0$ and $T>0$. Then, the error estimates
\begin{align}
&\|E_h(t)x\|_{H_0}\leq C_\epsilon h^\beta \|A^{\beta(\frac{1+\epsilon}{2})}x\|_{H_0},\quad \beta\in [0,2],\, x\in \mathcal{D}(A^{\beta(\frac{1+\epsilon}{2})}),\,t\in [0,T];\label{eq:esti1}\\
&\|E_h(t)\|_{\mathcal{L}(H_0)}\leq Ch^{2\beta}t^{-\beta(\alpha+1)},\quad \beta\in [0,1],\,t\in (0,T];\label{eq:esti2}\\\
&\|\dot{E}_h(t)\|_{\mathcal{L}(H_0)} \leq Ch^{2\beta}t^{-\beta(\alpha+1)-1},\quad \beta\in [0,1],\,t\in (0,T],\label{eq:esti3}
\end{align}
hold for $0<h< 1$ and $C,C_\epsilon>0$.
\end{proposition}
\begin{proof}
The error bound \eqref{eq:esti1} is shown in \cite[Proposition 3.3]{mishi1}. The error estimate \eqref{eq:esti2} for $\beta=1$ is proved in \cite[Theorem 2.1]{LST} while for $\beta\in [0,1)$ it follows immediately using also the stability estimate $\|E_h(t)\|\leq C$; the latter is a consequence of  \eqref{eq:esti1} with $\beta=0$. Thus, we have to prove \eqref{eq:esti3}.
It is shown in \cite[Equation (2.2)]{LST} that the Laplace transform $\widehat{E}(z)$ of $E_h$ satisfies the error estimate
\begin{equation}\label{eq:Ez}
\|\widehat{E}(z)\|_{\mathcal{L}(H_0)}\leq C h^2 |z|^\alpha
\end{equation}
in a symmetric sectorial region containing the right half-plane. We write
$$
E_h(t)=S(t)-S_h(t)P_h=S(t)P_h-S_h(t)P_h+S(t)(I-P_h):=E^1_h(t)+E^2_h(t).
$$
Let $x\in \mathcal{D}(A)$. Then, it follows that  $t\mapsto E^1_h(t)x$ is continuously differentiable on $[0,\infty)$, $E^1_h(0)x=0$ and
$$
\|\dot{E}^1_h(t)x\|_{H_0}\leq Ct^{\alpha}(\|Ax\|_{H_0}+\|A_hP_hx\|_{H_0}).
$$
Hence $t\mapsto \dot{E}^1_h(t)x$ is Laplace transformable and
\begin{equation}\label{eq:E1dd}
\widehat{\dot{E}^1_h}(z)x=z\widehat{E^1_h}(z)x=z\widehat{E}(z)P_hx.
\end{equation}
Let $\theta\in (\frac{\pi}{2},\frac{\pi}{1+\alpha})$ be fixed and let $\Gamma:=\{z:\, |\arg(z)|=\theta\}$ denote the curve with $\text{Im}\,z$ running from $-\infty$ to $\infty$. Then, using \eqref{eq:Ez} and \eqref{eq:E1dd}, we get
\begin{align*}
&\|\dot{E}^1_h(t)x\|_{H_0}=\left\|\frac{1}{2\pi i}\int_\Gamma e^{tz}z\widehat{E}(z)P_hx\,dz\right\|_{H_0}
\\
&\leq Ch^2 \int_\Gamma |z|^{\alpha+1} e^{-ct|z|}\, |dz|\|P_hx\|_{H_0}\leq  Ch^2 \int_0^\infty s^{\alpha+1}e^{-cst}\,ds\|P_h\|_{\mathcal{L}(H_0)}\|x\|_{H_0}\\
&\leq Ch^2 \int_0^\infty \left(\frac{r}{t}\right)^{\alpha+1}e^{-cr}\,\frac{dr}{t}\|x\|_{H_0} \leq C h^2 t^{-(\alpha+1)-1}\|x\|_{H_0},\,t>0.
\end{align*}
Therefore, as $\mathcal{D}(A)$ is dense in $H_0$, we conclude that
$$
\|\dot{E}^1_h(t)\|_{\mathcal{L}(H_0)}\leq C h^2 t^{-(\alpha+1)-1},\,t>0.
$$
To bound $\dot{E}^2_h(t)$ recall that $\|(I-P_h)x\|_{H_0}\leq Ch^2\|Ax\|_{H_0}$. Hence, using the smoothing property \eqref{eq:smp2} and the self-adjointness of $P_h$ and $\dot{S}$, we get
\begin{multline*}
\|\dot{E}^2_h(t)\|_{\mathcal{L}(H_0)}=\|\dot{S}(t)(I-P_h)\|_{\mathcal{L}(H_0)}=\|[\dot{S}(t)(I-P_h)]^*\|_{\mathcal{L}(H_0)}\\=\|(I-P_h)^*\dot{S}(t)^*\|_{\mathcal{L}(H_0)}
=\|(I-P_h)\dot{S}(t)\|_{\mathcal{L}(H_0)}\leq Ch^2\|A\dot{S}(t)\|_{\mathcal{L}(H_0)}\leq C h^2 t^{-(\alpha+1)-1},
\end{multline*}
where $L^*$ denotes the adjoint of an operator $L\in \mathcal{L}(H_0)$.
Thus, in summary,
$$
\|\dot{E}_h(t)\|_{\mathcal{L}(H_0)}\leq \|\dot{E}^1_h(t)\|_{\mathcal{L}(H_0)}+\|\dot{E}^2_h(t)\|_{\mathcal{L}(H_0)}\leq C h^2 t^{-(\alpha+1)-1},\,t>0,
$$
which is \eqref{eq:esti3} for $\beta=1$. Then, it follows that to show \eqref{eq:esti3} for $\beta\in [0,1]$ it is enough to prove that
$$
\|\dot{E}_h(t)\|_{\mathcal{L}(H_0)}\leq Ct^{-1},\, t>0.
$$
As, by \eqref{eq:smp2}, $\|\dot{S}(t)\|_{\mathcal{L}(H_0)}\leq Ct^{-1}$, $t>0$, we only need to prove that
\begin{equation}\label{eq:shd}
\|\dot{S}_h(t)P_h\|_{\mathcal{L}(H_0)}\leq Ct^{-1},\,t>0.
\end{equation}
It is well-known, see, for example \cite[Chapter 6]{Thomeebook} that the uniform resolvent estimate
\begin{equation}\label{eq:resolh}
\|(zI+A_h)^{-1}P_h\|_{\mathcal{L}(H_0)}\leq \frac{M_\omega}{|z|}
\end{equation}
holds in any sector $\Sigma_\omega=\{z:\, |\arg z|<\omega\}\setminus\{0\}$, $\omega\in (0,\pi).$ A simple calculation shows that
$$
\widehat{S_h}(z)P_h=z^{\alpha}(z^{1+\alpha}I+A_h)^{-1}P_h.
$$
Note that
$$
\|\dot{S}_h(t)P_hx\|_{H_0}\leq Ct^{\alpha}\|A_h\|_{\mathcal{L}(H_0)}\|x\|_{H_0}
$$
and thus $t\to\dot{S}_h(t)P_h$ is Laplace transformable and
$$
\widehat{\dot{S}_h}(z)P_h=z^{1+\alpha}(z^{1+\alpha}I+A_h)^{-1}P_h-P_h\fahimm{,}
$$
for all $z\in \Sigma_\omega$ with $\omega<\pi/(1+\alpha)$. Using \eqref{eq:resolh} it follows that
$$
\|\widehat{\dot{S}_h}(z)P_h\|_{\mathcal{L}(H_0)}\leq M
$$
for all $z\in \Sigma_\omega$ with $\omega<\pi/(1+\alpha)$. Hence, with $\Gamma$ as above, we have
$$
\|\dot{S}_h(t)P_h\|_{\mathcal{L}(H_0)}=\left\|\frac{1}{2\pi i}\int_\Gamma e^{tz}\widehat{\dot{S}_h}(z)P_h\,dz\right\|_{\mathcal{L}(H_0)}\leq M\int_\Gamma e^{-ct|z|}\, |dz|\leq \tilde{M} t^{-1},\,t>0,
$$
and the proof is complete.
\end{proof}
We can now prove an error estimate in H\"older norms.
\begin{theorem}\label{thm:femholder}
Let $U$ and $U_h$ be given by \eqref{eq:UUU} and \eqref{eq:uh}, respectively. Set $\Err(t) := U(t)-U_h(t)$ and $e_0(t):=S(t)U_0-S_h(t)P_hU_0$.
\begin{itemize}
\item[(a)]Let $p>2$ and $0< \gamma_1<\frac12-\frac{1}{p}$. Then, for $\gamma<\frac12-\frac{1}{p}-\gamma_1$ and $\beta<\frac{\gamma_1}{\alpha+1}$ the error estimate
\begin{equation}\label{eq:femest1}
 \|e\|_{L^p(\Omega;C^\gamma([0,T];H_0))}\leq C(T,p,\beta,\gamma)\lk( \|e_0\|_{L^p(\Omega;C^\gamma([0,T];H_0))} +\|e(0)\|_{L^p(\Omega;H_0)}+Ch^{2\beta}\rk)
\end{equation}
holds. If the mesh is quasi-uniform and $U_0\in L^p(\Omega;D(A^{\frac{1}{1+\alpha}}))$, then
\begin{equation}\label{eq:femest2}
 \|e\|_{L^p(\Omega;C^\gamma([0,T];H_0))}\leq C(T,p,\beta,\gamma,U_0)h^{2\beta}.
\end{equation}
\item[(b)] Let $p>2$. Then, for $\beta<\frac{\frac12-\frac{1}{p}}{\alpha+1}$ the error estimate
\begin{equation}\label{eq:femest1a}
 \|e\|_{L^p(\Omega;C([0,T];H_0))}\leq C(T,p,\beta,\gamma)\lk( \|e_0\|_{L^p(\Omega;C([0,T];H_0))} +Ch^{2\beta}\rk)
\end{equation}
holds. If  $U_0\in L^p(\Omega;D(A^{\beta(1+\epsilon)}))$ for some $\epsilon>0$, then
\begin{equation}\label{eq:femest2a}
 \|e\|_{L^p(\Omega;C([0,T];H_0))}\leq C(T,p,\beta,\gamma,U_0)h^{2\beta}.
\end{equation}
\end{itemize}
\end{theorem}
\begin{proof}
Let $0< \gamma_1<\frac12-\frac{1}{p}$. Then,
\begin{equation}\label{eq:tga1}
t^{\gamma_1}\|\dot{S}(t)x\|_{H_0}+t^{\gamma_1-1}\|S(t)x\|_{H_0}\leq Ct^{\gamma_1-1}\|x\|_{H_0}:=s_1(t)\|x\|_{H_0}
\end{equation}
and thus $s_1\in L ^1 ((0,T];\RR ^+_0)$. \fahimm{Furthermore, by Proposition \ref{prop:estimateEh},}
\begin{equation}\label{eq:tga2}
t^{\gamma_1}\|\dot{E}_h(t)x\|_{H_0}+t^{\gamma_1-1}\|E_h(t)x\|_{H_0}\leq Ch^{2\beta}t^{-{\beta}(\alpha+1)-1+\gamma_1}\|x\|_{H_0}:=h_1(t)h^{2\beta}\|x\|_{H_0},
\end{equation}
for $0<h<1$. We have that $h_1\in L ^1 ((0,T];\RR ^+_0)$ if and only if $-{\beta}(\alpha+1)-1+\gamma_1>-1$; that is, when $\beta<\frac{\gamma_1}{\alpha+1}$. Then, the error bound in \eqref{eq:femest1} follows from Theorem \ref{thm:main} for each fixed $0<h<1$ with $\Psi_{n}^1=h^{-2\beta}E_{h}$, $r_1(n)=h^{2\beta}$ for all $n\in \mathbb{N}$ and $\Psi_n^2=0$, $r_2(n)=h^{2\beta}$ for all $n\in \N$ using \eqref{eq:tga1} and \eqref{eq:tga2} and noting that the function $h_1$ in \eqref{eq:tga2} is independent of $h$ and then so is the constant $C$ in the error estimate \eqref{eq:mee} of Theorem \ref{thm:main}. To show \eqref{eq:femest2} note that, by a standard finite element estimate, we have
\begin{equation}\label{eq:femin}
\|e(0)\|_{L^p(\Omega;H_0)}=\|(I-P_h)U_0\|_{L^p(\Omega;H_0)}\leq Ch^{\frac{2}{1+\alpha}}\|A^{\frac{1}{1+\alpha}}U_0\|_{L^p(\Omega;H_0)}.
\end{equation}
Furthermore,
for $x\in H_0$, by choosing $\chi=P_hx$ in \eqref{eq:vstroh}, we see that the function $t\mapsto S_h(t)P_hx$ is the unique solution of  \eqref{eq:vstroh}; that is,
$$
\dot{S}_h(t)P_hx+\frac{1}{\Gamma(\alpha)}\int_0^t (t-s)^{\alpha-1}A_hS_h(t)P_hx\,d s=0,\,t>0,
$$
and therefore it follows that
$$
\|\dot{S_h}(t)P_hx\|_{H_0}\leq Ct^\alpha \|A_hP_hx\|_{H_0},
$$
where we used the stability estimate $\|S_h(t)P_h\|_{\mathcal{L}(H_0)}\leq C$, $t\geq 0$. Using also \eqref{eq:shd} we conclude, by interpolation that
\begin{equation}\label{eq:shddd}
\|\dot{S}_h(t)P_hx\|_{H_0}\leq Ct^{\mu(1+\alpha)-1}\|A_h^{\mu}P_hx\|_{H_0}, \,\mu\in [0,1].
\end{equation}
Therefore, using \eqref{eq:smp2} and  \eqref{eq:shddd}  with $\mu=\frac{1}{1+\alpha}$, it follows that
\begin{align*}
&\|\dot{E}_h(t)x\|_{H_0}\leq \|\dot{S}_h(t)P_hx\|_{H_0}+ \|\dot{S}(t)x\|_{H_0}\leq \fahimm{C}\|A_h^{\frac{1}{1+\alpha}}P_hx\|_{H_0}+\fahimm{C}\|A^{\frac{1}{1+\alpha}}x\|_{H_0}\\
&\quad= \fahimm{C}\|A_h^{\frac{1}{1+\alpha}}P_hA^{-\frac{1}{1+\alpha}}A^{\frac{1}{1+\alpha}}x\|_{H_0}+\fahimm{C}\|A^{\frac{1}{1+\alpha}}x\|_{H_0} \\
&\quad \leq C\|A_h^{\frac{1}{1+\alpha}}P_hA^{-\frac{1}{1+\alpha}}\|_{\mathcal{L}(H_0)}\| A^{\frac{1}{1+\alpha}}x\|_{H_0} +\fahimm{C}\|A^{\frac{1}{1+\alpha}}x\|_{H_0}\leq C \|A^{\frac{1}{1+\alpha}}x\|_{H_0}, \,t\in [0,T].
\end{align*}
Here we also used that, by the self-adjointness of $A^{-1},A_h$ and $P_h$,
\begin{align*}
\|A_h^{\delta}P_hA^{-\delta}\|_{\mathcal{L}(H_0)}&=\|P_hA_h^{\delta}P_hA^{-\delta}\|_{\mathcal{L}(H_0)}=\|(P_hA_h^{\delta}P_hA^{-\delta})^{*}\|_{\mathcal{L}(H_0)}\\
&=\|(A^{-\delta})^*(P_hA_h^{\delta}P_h)^*\|_{\mathcal{L}(H_0)}=\|A^{-\delta}A_h^{\delta}P_h\|_{\mathcal{L}(H_0)}\leq C,
\end{align*}
where the last inequality holds for $\delta\in [0,1]$ for quasi-uniform meshes, see, for example, the proof of Theorem 4.4 (iv) in \cite{weak}. Thus, using interpolation in H\"older spaces and the smooth data estimate \fahim{from}
\eqref{eq:esti1},
\begin{equation*}
\|e_0\|_{C^\gamma([0,T];H_0)}\leq C \|e_0\|_{C^1([0,T];H_0)}^{\gamma}\|e_0\|_{C([0,T];H_0)}^{1-\gamma}\leq C h^{(1-\gamma)\frac{1}{1+\alpha}\frac{2}{1+\epsilon}}\|A^{\frac{1}{1+\alpha}}U_0\|_{H_0}.
\end{equation*}
As
$$
(1-\gamma)\frac{1}{1+\alpha}\frac{2}{1+\epsilon}>\frac{2\gamma_1}{1+\epsilon}\frac{1}{1+\alpha}+\left(\frac{1}{p}+\frac{1}{2}\right)\frac{1}{1+\alpha}\frac{2}{1+\epsilon}>\frac{2\gamma_1}{1+\epsilon}\frac{1}{1+\alpha}
$$
it follows that
$$
\|e_0\|_{L^p(\Omega;C^\gamma([0,T];H_0))}\leq Ch^{2\beta} \|A^{\frac{1}{1+\alpha}}U_0\|_{L^p(\Omega;H_0)}
$$
for all $\beta<\frac{\gamma_1}{\alpha+1}$ and the proof of \eqref{eq:femest2} is complete in view of \eqref{eq:femest1} and \eqref{eq:femin}.
Next, the estimate  \eqref{eq:femest1a} follows from Corollary \ref{cor:main} in view of \eqref{eq:tga1} and \eqref{eq:tga2}. Finally, using \eqref{eq:esti1}, we immediately conclude that
$$
\|e_0\|_{L^p(\Omega;C([0,T];H_0))}\leq Ch^{2\beta} \|A^{\beta(1+\epsilon)}U_0\|_{L^p(\Omega;H_0)},
$$
which finishes the proof of \eqref{eq:femest2a} in view of \eqref{eq:femest1a} and the proof of the theorem is complete.
\end{proof}
\begin{remark}
If $U_0$ is deterministic, then we may take $p$ arbitrarily large in Theorem \ref{thm:femholder} (similarly, in Theorem \ref{thm:example1} in case $u_0$ and $u_1$ are deterministic). We also point out that the estimate on $\|e_0\|_{C^\gamma([0,T];H_0)}$ in the proof of Theorem \ref{thm:femholder} is not sharp in terms of the regularity of the initial data. This follows from the fact that we estimate the $\gamma$-H\"older norm by interpolation and not directly and hence more regularity on $U_0$ is assumed than what is necessary. However, a sharp, direct estimate on $\|e_0\|_{C^\gamma([0,T];H_0)}$ is not available in the finite element literature, and a derivation would be beyond the scope of this paper.
\end{remark}
\begin{remark}[Stochastic heat equation]\label{rem:she} Here we briefly comment on the stochastic heat equation which also fits in our abstract framework. Suppose that $F$ and $G$ are as above and $S(t):=e^{-tA}$ is the heat semigroup and $S_h(t):=e^{-tA_h}P_h$, $t\geq 0$. In this case the well-known error estimates, see \cite[Chapter 3]{Thomeebook},
\begin{align}
&\|E_h(t)x\|_{H_0}\leq Ch^\beta \|A^{\frac{\beta}{2}}x\|_{H_0},\quad \beta\in [0,2],\, x\in \mathcal{D}(A^{\beta}),\,t\in [0,T];\label{eq:esti11}\\
&\|E_h(t)\|_{\mathcal{L}(H_0)}\leq Ch^{2\beta}t^{-\beta},\quad \beta\in [0,1],\,t\in (0,T];\label{eq:esti12}\\\
&\|\dot{E}_h(t)\|_{\mathcal{L}(H_0)} \leq Ch^{2\beta}t^{-\beta-1},\quad \beta\in [0,1],\,t\in (0,T],\label{eq:esti13}
\end{align}
hold for $0<h< 1$. Note that these are essentially \eqref{eq:esti1}-\eqref{eq:esti3} for $\alpha=0$. Then, similarly as in the proof of Theorem \ref{thm:femholder} we get, for $p>2$, $0< \gamma_1<\frac12-\frac{1}{p}$, $\gamma<\frac12-\frac{1}{p}-\gamma_1$ and $\beta<\gamma_1$ that the error estimate
\begin{equation}\label{eq:femest3}
 \|e\|_{L^p(\Omega;C^\gamma([0,T];H_0))}\leq C(T,p,\beta)\lk( \|e_0\|_{L^p(\Omega;C^\gamma([0,T];H_0))} +\|e(0)\|_{L^p(\Omega;H_0)}+Ch^{2\beta}\rk)
\end{equation}
holds, where $\Err(t) := U(t)-U_h(t)$. In particular we have
$$
 \|e\|_{L^p(\Omega;C([0,T];H_0))}\leq C(T,p,\beta)\lk( \|U_0\|_{L^p(\Omega;A^{\beta})} +1\rk)h^{2\beta}
$$
for $\beta<\frac12-\frac1{p}$. This result is consistent with \cite[Proposition 4.2]{sonja3} but less smoothness on the noise is assumed here; that is, we may take $\delta_G=0$.

\end{remark}
\begin{remark}
In \cite{mishi1}, a simplified version of \eqref{eq:fwe} was considered with $\Gamma(u)=I$ and $F=0$ (linear equation, additive noise). It was shown there that if $Q$ has finite trace then
$$
\sup_{t\in [0,T]}\|e(t)\|_{L^2(\Omega;H_0)}\leq C\lk( \|U_0\|_{L^2(\Omega;A^{\frac{1+\epsilon}{2(1+\alpha)}})} +1\rk)h^{\frac{1}{1+\alpha}}.
$$
This is consistent with  Theorem \ref{thm:femholder} as in this case  we may first take $U_0=0$  and hence take $p$ in \eqref{eq:femest1a} arbitrarily large and then add the estimate for the initial term.
\end{remark}

\section{Numerical experiments}\label{sec:ne}
In this section, we will illustrate our theoretical results by some numerical experiments. The underlying equation we consider is the fractional stochastic wave equation \eqref{eq:fwe}, where $\mathcal{D}=[0,1]$, $F=0$, $\Gamma(U)=I$, and $A=-\Delta$ is the Laplacian with Dirichlet boundary conditions in $H_0=(L^2(\mathcal{D}),\|\cdot\|)$ with inner product denoted by $\inner{\cdot}{\cdot}$.
In particular, we will implement the numerical solution for the following equation:
\begin{equation}\label{eq:num1}
\left\{
\begin{aligned}
dU(t,x)-\int_0^tb(t-s)\Delta U(s,x)\,ds\,dt&= Q^{\frac12}d \fahim{W_H}(t,x),\,t\in(0,1],\,x\in \mathcal{D}; \\
U(0,x)&=\sin(\pi x):=U_0(x),\,x\in \mathcal{D},
\end{aligned}
\right.
\end{equation}
\fahim{where $W_H$ is a
$H$-cylindrical Wiener process with $H=H_0$, $b(t)=t^{\alpha-1}/\Gamma(\alpha),$ $\alpha \in (0,1)$}, and $Q:H\to H$ is symmetric, bounded, and positive semidefinite. 

\medskip

In \fahim{Subsection} \ref{sec:numerical:MLEI}, we apply the spectral Galerkin method based on the eigenvalues $\lambda_k=k^2\pi^2$ and the orthonormal basis of corresponding eigenfunctions $\{e_k:k\in \NN  \}$. For the driving noise we take space time white noise; that is, $Q=I$. In particular, we take \fahim{$W_H$ to be given by the formal series \(W_{\fahimm{H}}(t,x)=\sum_{k=1}^{\infty} e_k(x)\beta^k(t) \),} $x\in \mathcal{D}$, $t\geq 0$,
where $\{\beta^k:k=1,2,\dots\}$ is a family of  mutually independent standard scalar Brownian motions.
To perform the integration in time, we use the Mittag–Leffler Euler integrator (MLEI) method, developed for semilinear problems in \cite{KLS20}. In the present linear setting this method is exact, that is, no additional time-discretization error is introduced and we may simulate the spatially approximated process exactly on a time-grid.

In \fahim{Subsection} \ref{sec:numerical:LCQM}, we approximate the solution of \eqref{eq:num1} by finite elements.
We consider a Wiener process which is of trace class given by
\begin{equation}\label{eq:wqtr}
Q^{\frac12}\fahim{W_H(t,x)}:=1_{[0,0.5]}(x)\beta(t), \fahim{x\in \mathcal{D}, t\in [0,1],}
\end{equation}
where $\beta$ is a scalar Brownian motion and $1_{[0,0.5]}$ is the characteristic function of the interval $[0,0.5]$. That is, the Fourier expansion of the driving Wiener process contains a single term only and thus its covariance operator is of rank 1 and hence trace class. The motivation for the particular choice of the Wiener process is to consider trace class noise which does not possess additional spatial smoothness. This is needed so that we do not observe higher convergence rate, due to additional regularity, in the numerical experiments than predicted by the theory for trace class noise. 
To perform the  time integration we implement a Lubich Convolution Quadrature (LCQ) method, for details see  \cite{Lubich-1,Lubich-2}.
This method was successfully applied to a similar problem of the third author in \cite{mishi1}.
The LCQ method is easier to implement than the  MLEI in case of  finite elements and a correlated  noise.

\subsection{The spectral Galerkin method and the MLEI-method}\label{sec:numerical:MLEI}
The mild solution of \eqref{eq:num1} with space time white noise can be written as 
\begin{equation}\label{eq:num_mild}
U(t)=S(t)U_0+\sum_{k=1}^{\infty} \int_0^tS(t-\tau)  e_k ~d\beta^k(\tau),
\end{equation}
where, as shown in   \cite{KLS20},
the resolvent family $\{S(t)\}_{t\geq 0}$ can be represented as
\begin{equation}\label{num_2}
S(t)v=\sum\limits_{k=1}^{\infty} E_{\alpha+1}(-\lambda_kt^{\alpha+1})(v,e_k)e_k,\, t>0,
\end{equation}
where $E_{\rho}(z)$, $\rho>0$, is the one parameter \fahim{Mittag-Leffler} function (MLF) defined by
$$
E_{\rho}(z):=\sum\limits_{k=0}^{\infty} \frac{z^k}{\Gamma(\rho k+1)},\quad z\in \CC.
$$
{For more details about Mittag-Leffler function and their application, we refer to the paper \cite{Leffler1903} and the book \cite{Gorenflo2010}.} Moreover, in order to implement the MLF, we use the Matlab function mlf.m, see \cite{Pod2009}.

Let $\Pi=\{0=t_0<t_1<\cdots<t_M=1\}$ be a partition of the time interval $[0,1]$. {From the representation given in \eqref{eq:num_mild} we get for $m=0,1,2,3,\dots,M$}
\begin{equation}\label{eq:num_mild2}
U(t_m)=S(t_m)U_0+\sum_{k=1}^{\infty} \int_0^{t_m}S(t_m-s)  e_k(x) ~d\beta^k(s).
\end{equation}
For the discretization in space, we introduce the finite dimensional subspaces $\fahim{H^N}=\text{span}\{e_k:k=1,2,\dots,N \}$ of $H$ and the orthogonal projection $\mathcal{P}_N:H\to \fahim{H^N}$ given by
$$\mathcal{P}_N v= \sum\limits_{k=1}^{N}(v,e_k)e_k,\quad v\in H.
$$
Using \eqref{num_2} we then get
$$
S_N(t) v:= S(t)\mathcal{P}_N v =\sum\limits_{k=1}^{N} E_{\alpha+1}(-\lambda_kt^{\alpha+1})(v,e_k)e_k.
$$
This way we obtain for the approximation $\bar U^N_m$ of $U(t_m)$ given by \eqref{eq:num_mild} by the Galerkin method
\begin{equation}\label{eq:num_mild3}
\bar U_m^N=S_N(t_m)\bar U_0^N+\sum_{k=1}^{N} \int_0^{t_m}S_N(t_m-s) e_k ~d\beta^k(s),
\end{equation}
with initial value $\bar U_0^N=\mathcal{P}_N U_0$.
Let us define $\bar U_{m,k}^N$ by
\begin{equation}\label{eq:num_mild4}
\bar U_{m,k}^N=E_{\alpha+1}(-\lambda_kt_m^{\alpha+1}) \bar U^N_{0,k}+ \mathcal{O}_k(t_m),
\end{equation}
where $\bar U^N_{0,k}=(U(0),e_k)$ and
$$
\mathcal{O}_k(t_m):=\int_0^{t_m}E_{\alpha+1}(-\lambda_k(t_m-s)^{\alpha+1}) ~d\beta^k(s).
$$
Then, \eqref{eq:num_mild3} can be rewritten as
$$\bar U_m^N=\sum\limits_{k=1}^{N} \bar U_{m,k}^{N}e_k.
$$
%
To simulate the stochastic convolution process let us observe  that $$\mathcal{N}:= (\mathcal{O}_k(t_1),\mathcal{O}_k(t_2),\dots,\mathcal{O}_k(t_M))^\top$$ is a $M$-dimensional Gaussian random variable with zero mean and covariance matrix $ R=(R_{i,j}  )_{i,j=1}^{M}$
$$R_{i,j}=\int_0^{t_i\wedge t_j}E_{\alpha+1}(-\lambda_k(t_i-s)^{\alpha+1})E_{\alpha+1}(-\lambda_k(t_j-s)^{\alpha+1}) ~ds.$$
Thus, $\mathcal{N}$ can be represented as $K\chi$, where $\chi$ is an $M$-dimensional standard Gaussian random variable and $K$ is the solution of equation $KK^T=R$ (see Theorem 2.2 of \cite{Gut2009});  the equation $KK^T=R$ can be solved by the Cholesky factorization.

\begin{figure}[!ht]
	\centering
	\includegraphics[width=0.65\linewidth]{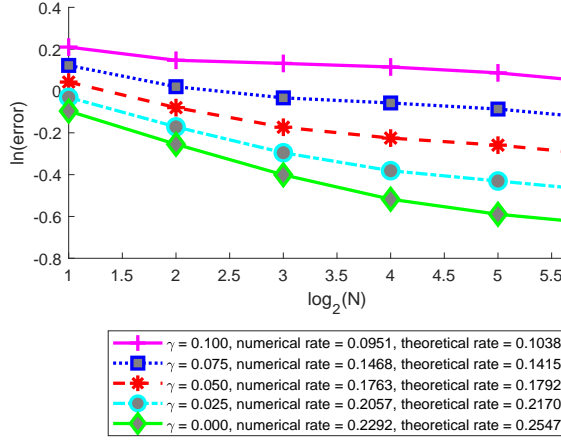}
	\caption{The approximation error for the spectral Galerkin method and the MLEI-method in the $L^2(\Omega;C^\gamma([0,T];H_0))$-norm with $\alpha=0.325$.}
	\label{fig:GA1}
\end{figure}

In our numerical experiment, we simulate $100$ sample paths to verify the  rate of convergence in the $L^2(\Omega;C^\gamma([0,T];H_0))$-norm
for different $\gamma$ and $\alpha$.   According to Example \ref{ex:sg3} we expect theoretical rate of 
$\nu<\frac{\gamma_1}{1+\alpha}-\delta_G$ in the $L^p(\Omega;C^\gamma([0,T];H_0))$-norm for appropriately smooth and integrable initial data $U_0$, where $\gamma<\frac{1}{2}-\frac1p-\gamma_1$ and $p>2$. Note that the parameter $\alpha$ in Example \ref{ex:sg3} corresponds to $\alpha+1$ in the present example, see Remark \ref{rem:corresp}.  Taking into account that $\lambda_N=N^2\pi^2$ and that $Q=I$ and hence $\delta_G>\frac{1}{4}$ we obtain
a rate in $N$ of almost $2(\tfrac{1}{2}-\tfrac1p-\gamma)/(1+\alpha)-\frac12$.
 Note that since $U_0$ is a deterministic eigenfunction of $A$ and thus $U_0\in L^p(\Omega;\mathcal{D}(A^s))$ for any $p>2$ and $s\geq 0$, we may bound the $L^2(\Omega;C^\gamma([0,T];H_0))$-norm by the $L^p(\Omega;C^\gamma([0,T];H_0))$-norm for any $p>2$ and hence we expect a rate in $N$ of almost $(1-2\gamma)(1+\alpha)-\frac12$ in the $L^2(\Omega;C^\gamma([0,T];H_0))$-norm. In the simulations, we chose a small time step $\Delta t=t_k-t_{k-1}=0.001$, $k=0,1,\dots,M$ and vary the dimension of the finite dimensional approximation space $\fahim{H^{N_i}}$, $i=1,2,\dots,6$, with $N_i=2^i$. 
To estimate the error, we computed  a reference solution with \fahim{ $N=2^{13}$.} 
In Figure \ref{fig:GA1}, we present the error of the numerical approximation in the $L^2(\Omega;C^\gamma([0,T];H_0))$-norm for $\alpha=0.325$ with varying $\gamma$ (see also  Figure \ref{fig:GA2} and Figure \ref{fig:GA3} for $\alpha=0.35$ and $\alpha=0.375$, respectively). \fahim{In Figures \ref{fig:GA1}-\ref{fig:GA3}, we also compute the numerical rate of convergence given by
\begin{equation}\label{rate_numeric_formula}
\min_{i=1,2,3,4,5}-\frac{\ln\left( \frac{\text{error}_{\gamma}(N_i)}{\text{error}_{\gamma}(N_{i+1})} \right)}
{\ln\left( \frac{N_i}{N_{i+1}}\right)},
\end{equation}	
for $\gamma=0,0.025,0.05,0.075,0.1$ where $\text{error}_{\gamma}(N_i)$ is the error of the numerical approximation in the $L^2(\Omega;C^\gamma([0,T];H_0))$-norm when the dimension of $H^{N_i}$ is $N_i$.
}Here, one may observe that if $\gamma$ decreases, then the rate of convergence increases. Moreover, Figures \ref{fig:GA1}-\ref{fig:GA3} also show that the numerical rate of convergence is close to the theoretical rate.

\begin{figure}[!ht]
	\centering
	\includegraphics[width=0.6\linewidth]{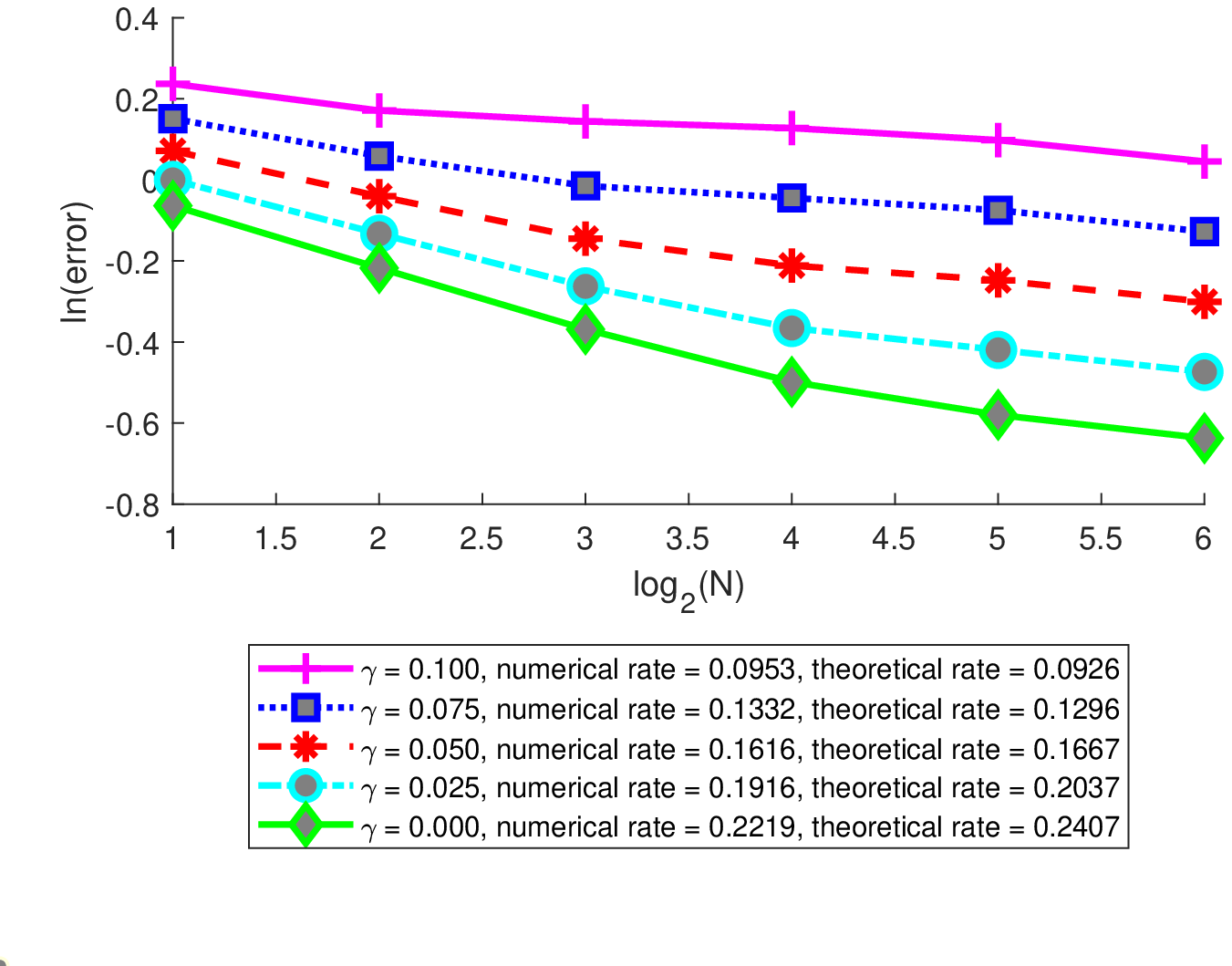}
	\caption{The approximation error for the spectral Galerkin method and the MLEI-method in the $L^2(\Omega;C^\gamma([0,T];H_0))$-norm with $\alpha=0.35$.}
	\label{fig:GA2}
\end{figure}

\begin{figure}[!ht]
	\centering
	\includegraphics[width=0.6\linewidth]{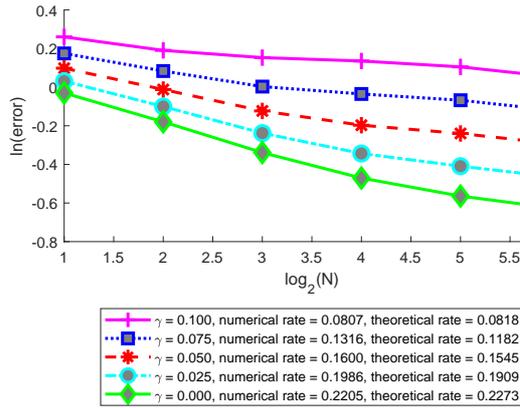}
	\caption{The approximation error for the spectral Galerkin method and the MLEI-method in the  $L^2(\Omega;C^\gamma([0,T];H_0))$-norm with $\alpha=0.375$.}
	\label{fig:GA3}
\end{figure}


\subsection{The finite element method and the LCQ-method}\label{sec:numerical:LCQM}

We first perform a time discretization of \eqref{eq:num1} with $Q^{\frac12}W_H$ given by \eqref{eq:wqtr} by the first order LCQ-method, for more details see e.g.\ \cite{mishi1}. To describe the first order LCQ method, let $\Pi=\{0=t_0<t_1<t_2,\cdots<t_M=1\}$ be an equidistant partition of the time interval $[0,1]$ with time step size $\Delta t=t_{m}-t_{m-1}$, $m=1,2,3,\dots, M$.  The approximation
of a convolution term
$$
\int_0^{t_m} b(t_m-s)g(s)\, ds
$$
is then given by
$$
\sum_{i=1}^m\omega_{m-i}g(t_i),
$$
where the weights $\{\omega_k:k\in\NN\cup \{0\}\}$
are chosen such that 
$$
\sum_{k=0}^\infty \omega_kz^k=\widehat b\lk(\frac {1-z}{\Delta t}\rk),\quad |z|<1.
$$
This is a first order quadrature; that is, it has an approximation order of $\mathcal{O}(\Delta t)$. Applying the  LCQ-method, the  equation for the approximation  $\bar U$, where  $\bar U_n(x)\approx  U(t_n,x)$, can be written  as follows
\begin{equation}\label{eq:numFE1}
\bar U_n-\bar U_{n-1}+\Delta t\Big(\sum\limits_{i=1}^{n}\omega_{n-i}A \bar U_i \Big)=1_{[0,0.5]} \Delta_n \beta,\quad \bar U_0=U(0),
\end{equation}
where  $\Delta_n \beta=\beta(t_n)-\beta(t_{n-1})$, $n=1,2,\dots, M$.

Secondly, we discretize \eqref{eq:numFE1} by linear finite elements. {Let us consider a partition of the domain} $\mathcal{D}=[0,1]$ given by $\{0=x_0<x_1<x_2<\cdots<x_N=1\}$ with constant mesh size $h=x_{m+1}-x_{m}$, $m=0,1,\dots, N-1$. Let us denote the  finite element spaces by $\{ V_h \}_{0<h<1}$, where $V_h=\text{span}\{\varphi_k:k=1,2,\dots,N-1 \}\subset \fahim{H^1_0(\mathcal{D})}$ with $\varphi_k$ being a standard hat function in the $1$-D finite element method \cite{Larson2010}. We introduce the discrete Laplacian
\begin{equation}\label{def:Ah2}
A_{h}:V_h\to V_h,   \quad
\inner{A_{h} \xi}{ \chi} = \inner{  \xi'}{\chi'},\quad \xi,\chi \in V_h,
\end{equation}
where $v'=\frac{dv}{dx}$ denotes the derivative, and the orthogonal projection
$$
P_{h}: H_0 \to V_h,\quad
\inner{P_{h} f}{ \chi} = \inner{ f} {\chi},\quad f\in H_0,~ \chi \in V_h.
$$
In order to obtain the numerical formulation for \eqref{eq:num1}, we compute a $V_h$-valued random variable $\bar U^h_{n}$ satisfying for all $k=1,2,\dots,N-1$
\begin{equation}\label{eq:numFE2}
\left\{
\begin{aligned}
(\bar U_n^h,\varphi_k)&= (\bar U_{n-1}^h,\varphi_k)- \Delta t\Big(\sum\limits_{i=1}^{n}\omega_{n-i}(A_h\bar  U^h_{i},\varphi_k)\Big)+ (1_{[0,0.5]} ,\varphi_k)\Delta_n \beta; \\
(\bar U^h_0,\varphi_k)&=(\sin(\pi x),\varphi_k),
\end{aligned}
\right.
\end{equation}
where $\bar U_n^h(x)=\sum\limits_{k=1}^{N-1} \bar U_{n,k}^h\varphi_k(x)\approx U(t_n,x)$.
From \eqref{def:Ah2} we then obtain
\begin{equation}\label{eq:numFE3}
\left\{
\begin{aligned}
\lefteqn{\sum\limits_{m=1}^{N-1} [(\varphi_m,\varphi_k)+\Delta t \omega_0 (\varphi'_m,\varphi'_k)] \bar  U_{n,m}^h}&\\
&= \sum\limits_{m=1}^{N-1} (\varphi_m,\varphi_k)\bar U_{n-1,m}^h -\Delta t  \sum\limits_{m=1}^{N-1} ( \varphi'_m,\varphi'_k) \Big(\sum\limits_{i=1}^{n-1}\omega_{n-i} \bar U_{i,m}^h \Big)+ (1_{[0,0.5]} ,\varphi_k)\Delta_n \beta; \\
&\sum\limits_{m=1}^{N-1} (\varphi_m,\varphi_k) \bar U_{0,m}^h =(\sin(\pi x),\varphi_k).
\end{aligned}
\right.
\end{equation}
The above system can be rewritten in the following form
\begin{equation}\label{eq:numFE4}
\boldsymbol{\bar U_n^h}= (\boldsymbol{K}+\Delta t \omega_0 \boldsymbol{L} )^{-1}(\boldsymbol{K} \boldsymbol{\bar U_{n-1}^h}-\Delta t \sum_{i=1}^{n-1}\omega_{n-i}\boldsymbol{L} \boldsymbol{\bar U_i^h}+\boldsymbol{J}\Delta_n\beta).
\end{equation}
Here, the vectors $\boldsymbol{\bar U_n^h}$ and $\boldsymbol{J}$ are defined by $\boldsymbol{\bar U_n^h}=(\bar U_{n,1}^h,\dots,\bar U_{n,N-1}^h)^\top
$ and $\boldsymbol{J}=(J_1, J_2,\dots, J_{N-1})^\top$ where $J_k=(1_{[0,0.5]} ,\varphi_k)$ for $k=1,\dots,N-1$. Moreover, the stiffness matrix $\boldsymbol K=(K_{i,j}  )_{i,j=1}^{N-1}$ and the mass matrix $\boldsymbol  L=(L_{i,j}  )_{i,j=1}^{N-1}$ are given by
$$K_{i,j}=\int\limits_{0}^{1}\varphi_i(x)\varphi_j(x)~dx,\quad L_{i,j}=\int\limits_{0}^{1}\varphi'_i(x)\varphi'_j(x)~dx,$$
respectively.

\begin{figure}[!ht]
	\centering
	\includegraphics[width=0.65\linewidth]{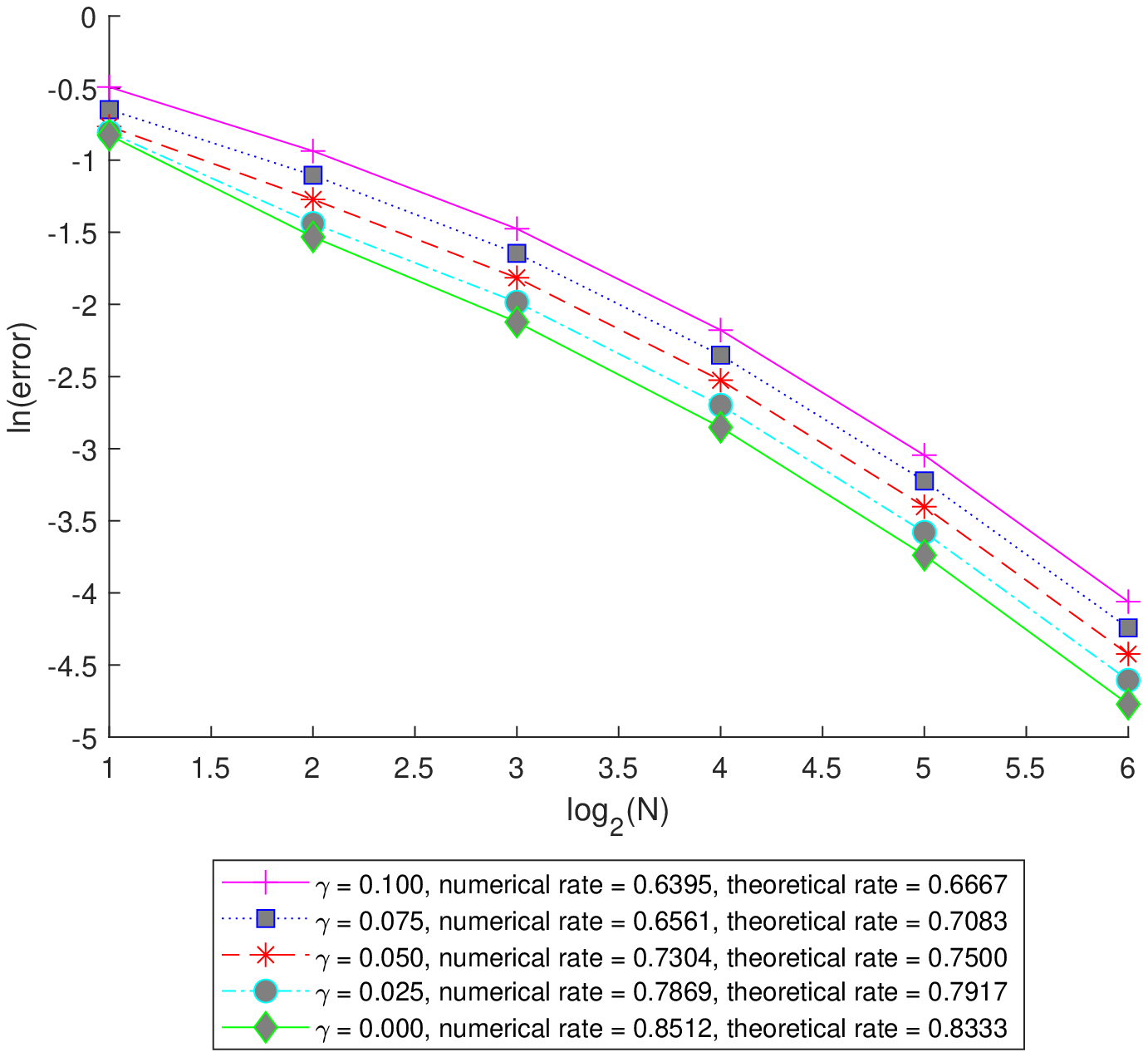}
	\caption{The approximation error for the finite element method and the LCQ-method in the $L^2(\Omega;C^\gamma([0,T];H_0))$-norm with $\alpha=0.2$.}
	\label{fig:FE1}
\end{figure}

In our numerical experiment, we used $500$ sample paths to verify the dependence of the rate of convergence in the $L^2(\Omega;C^\gamma([0,T];H_0))$-norm on $\gamma$ and $\alpha$. The  theoretical rate of convergence is almost   $(1-2\gamma)/(1+\alpha)$ 
according to Theorem \ref{thm:femholder}. Note again that, similarly to the previous example, since $U_0$ is a deterministic eigenfunction of $A$, we may bound the $L^2(\Omega;C^\gamma([0,T];H_0))$-norm by the $L^p(\Omega;C^\gamma([0,T];H_0))$-norm for any $p>2$. In the simulations  we choose fixed a step time $\Delta t=0.0005$ and varying the dimension of the space approximation $\dim V_h=N_i-1=2^i$, $i=1,2,\dots,6$; that is,  we take $h=\frac{1}{N_i}$, $i=1,2,\dots, 6$. Then, in order to measure the error, we computed a reference solution with \fahim{a mesh size $h=\frac{1}{2^{11}}$.} In Figure \ref{fig:FE1}, we present the error of the numerical approximation in the $L^2(\Omega;C^\gamma([0,T];H_0))$-norm for $\alpha=0.2$ with varying values of $\gamma$ (see also  Figure \ref{fig:FE2} and Figure \ref{fig:FE3} for $\alpha=0.25$ and $\alpha=0.3$, respectively). \fahim{Similarly to Section \ref{sec:numerical:MLEI}, in this section we also compute the numerical rate of convergence according to \eqref{rate_numeric_formula} for $\gamma=0,0.025,0.05,0.075,0.1$.}
Here, one may observe that if $\gamma$ decreases, then the rate of convergence again increases. Moreover, the figures also show that the numerical rate of convergence is close to the theoretical rate.

\begin{figure}[!ht]
	\centering
	\includegraphics[width=0.7\linewidth]{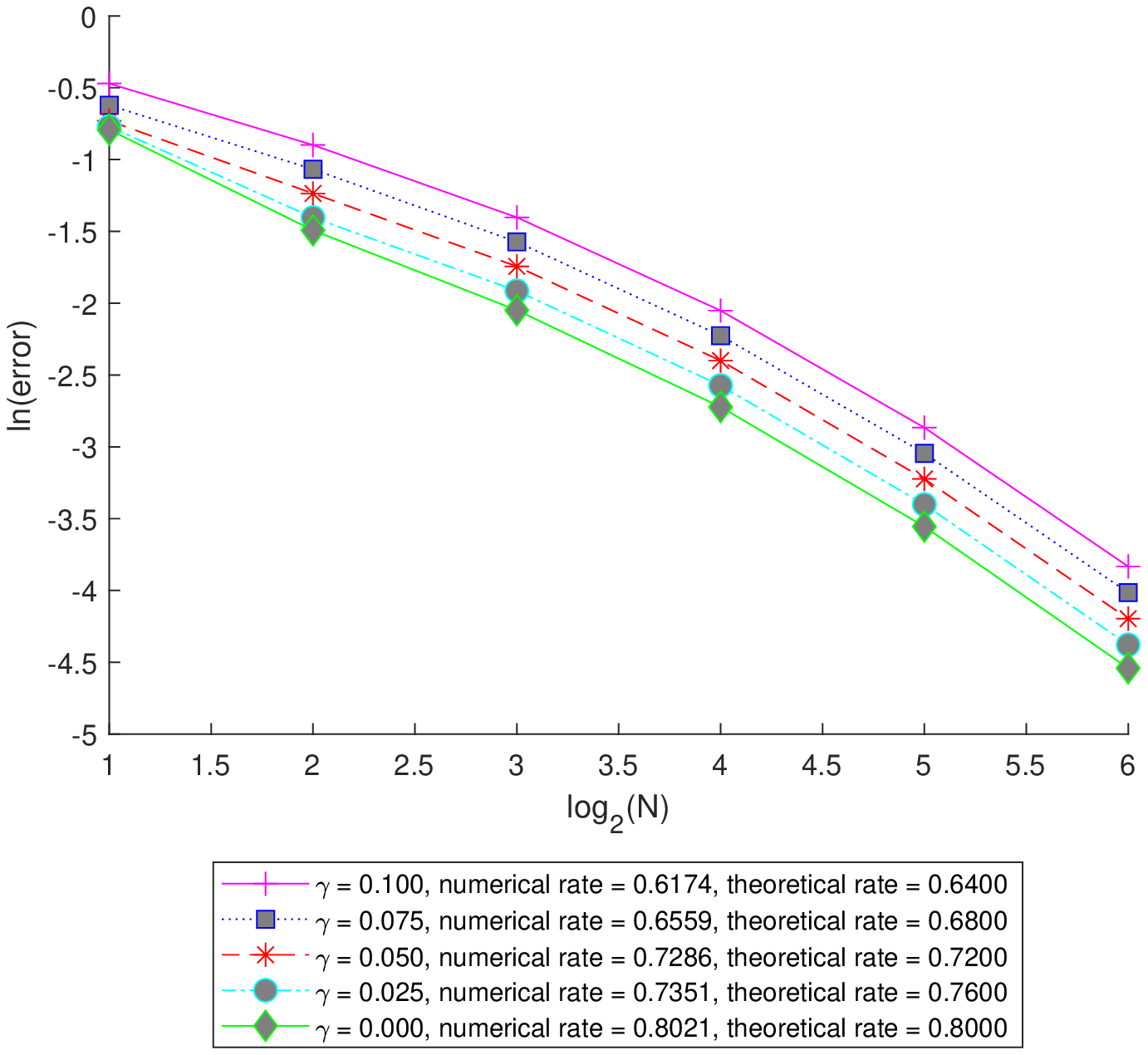}
	\caption{The approximation error for the finite element method and the LCQ-method in the $L^2(\Omega;C^\gamma([0,T];H_0))$-norm with $\alpha=0.25$.}
	\label{fig:FE2}
\end{figure}

\begin{figure}[!ht]
	\centering
	\includegraphics[width=0.7\linewidth]{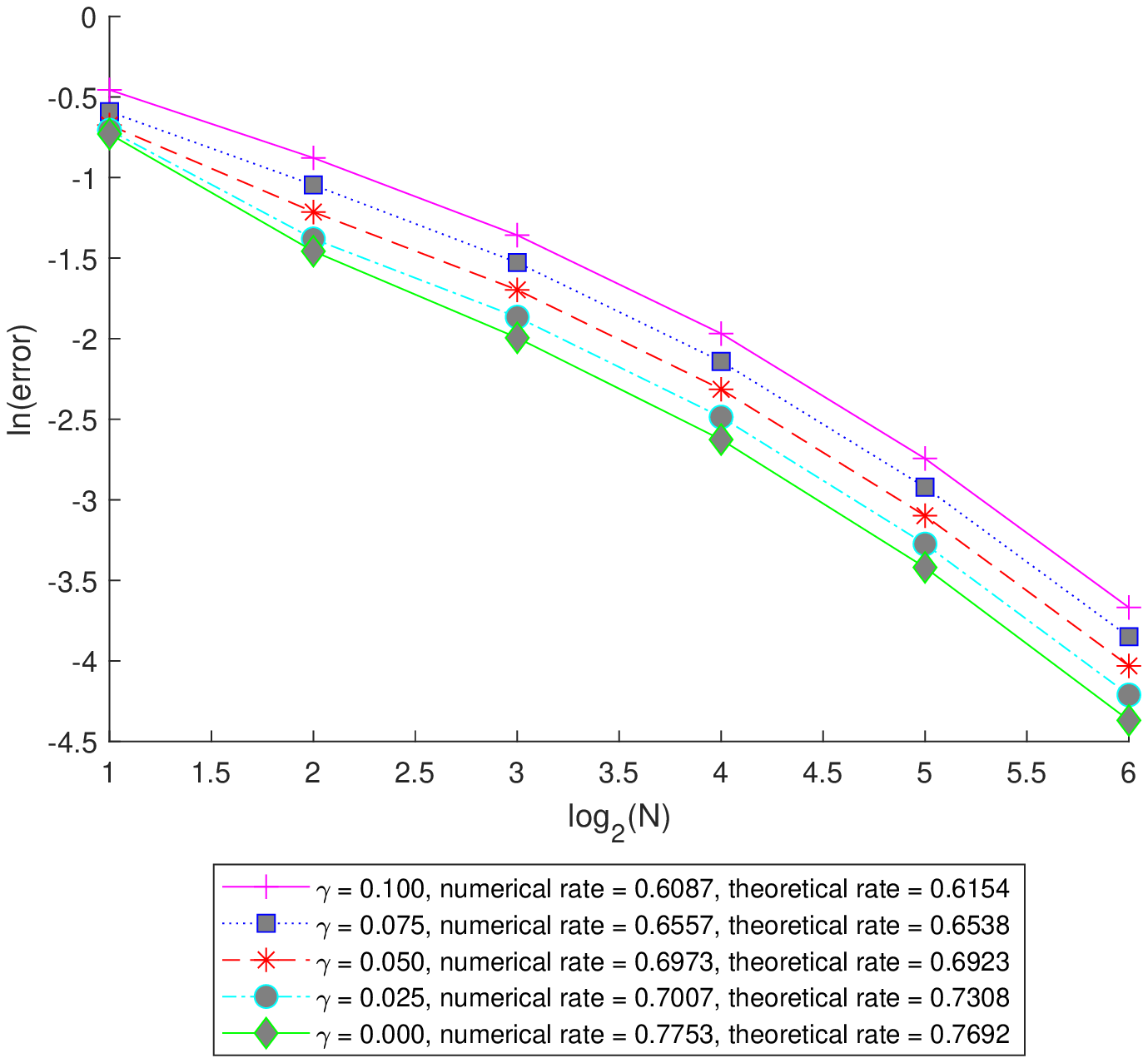}
	\caption{The approximation error for the finite element method and the LCQ-method in the $L^2(\Omega;C^\gamma([0,T];H_0))$-norm with $\alpha=0.3$.}
	\label{fig:FE3}
\end{figure}

\subsection*{Acknowledgement.} The authors would like to thank the anonymous referees for their careful reading of the manuscript and for their useful comments that
helped them to improve the presentation of the paper significantly.

\end{document}